\documentclass[12pt,english]{article}
\usepackage[T1]{fontenc}
\usepackage[latin9]{inputenc}
\usepackage{geometry}
\usepackage{amsmath}
\usepackage{amsthm}
\usepackage{amssymb}
\usepackage{graphicx}

\makeatletter
\numberwithin{equation}{section}
  \theoremstyle{plain}
  \newtheorem{thm}{\protect\theoremname}[section]
  \theoremstyle{remark}
  \newtheorem{rem}{\protect\remarkname}[section]
  \theoremstyle{definition}
  \newtheorem{defn}{\protect\definitionname}[section]
  \theoremstyle{plain}
  \newtheorem{lem}{\protect\lemmaname}[section]
  \theoremstyle{plain}
  \newtheorem{prop}{\protect\propositionname}[section]
  \theoremstyle{plain}
  \newtheorem{cor}{\protect\corollaryname}[section]

\date{}

\makeatother

\usepackage{babel}
  \providecommand{\definitionname}{Definition}
  \providecommand{\lemmaname}{Lemma}
  \providecommand{\propositionname}{Proposition}
  \providecommand{\remarkname}{Remark}
\providecommand{\corollaryname}{Corollary}
\providecommand{\theoremname}{Theorem}

\begin{document}

\title{
}
\title{Tail Asymptotics of the Brownian Signature}

\author{H. Boedihardjo\thanks{Department of Mathematics and Statistics , University of Reading,
Reading RG6 6AX, United Kingdom. Email: h.s.boedihardjo@reading.ac.uk.}, X. Geng\thanks{Department of Mathematical Sciences, Carnegie Mellon University, Pittsburgh
PA 15213, United States. Email: xig@andrew.cmu.edu.}}
\maketitle
\begin{abstract}
The signature of a path $\gamma$ is a sequence whose $n$-th term
is the order-$n$ iterated integrals of $\gamma$. It arises from
solving multidimensional linear differential equations driven by $\gamma$.
We are interested in relating the path properties of $\gamma$ with
its signature. If $\gamma$ is $C^{1}$, then an elegant formula of
Hambly and Lyons relates the length of $\gamma$ to the tail asymptotics
of the signature. We show an analogous formula for the multidimensional
Brownian motion, with the quadratic variation playing a similar role
to the length. In the proof, we study the hyperbolic development of
Brownian motion and also obtain a new subadditive estimate for the
asymptotic of signature, which may be of independent interest. As
a corollary, we strengthen the existing uniqueness results for the
signatures of Brownian motion. 
\end{abstract}

\section{\label{sec:intro}Introduction }

\subsection{Path driven differential equations and iterated integrals}

Path-driven differential equations of the form
\begin{equation}
dY_{t}=\sum_{i=1}^{d}A_{i}Y_{t}\mathrm{d}\gamma_{t}^{i},\;Y_{0}=y\label{eq:Sieltjes differential equation}
\end{equation}
where $\gamma=[0,T]\rightarrow\mathbb{R}^{d},\,\gamma=(\gamma^{1},\ldots,\gamma^{d})$
and $A_{i}:\mathbb{R}^{n}\rightarrow\mathbb{R}^{n}$ is linear, has
a Taylor expansion of the form 
\begin{equation}
Y_{T}=\sum_{n=0}^{\infty}\sum_{1\leq i_{1},\ldots,i_{n}\leq d}A_{i_{n}}A_{i_{n-1}}\ldots A_{i_{1}}\int_{0}^{T}\ldots\int_{0}^{t_{2}}\mathrm{d}\gamma_{t_{1}}^{i_{1}}\ldots\mathrm{d}\gamma_{t_{n}}^{i_{n}}.\label{eq:Taylor expansion}
\end{equation}
In particular, $Y_{t}$ is a linear function of the \textit{signature}
of $\gamma$ on $[0,T]$ (also known as the Chen series \cite{Chen54}),
defined as 
\begin{equation}
g\triangleq\left\{ \int_{0<t_{1}<\cdots<t_{n}<T}\mathrm{d}\gamma_{t_{1}}\otimes\cdots\otimes\mathrm{d}\gamma_{t_{n}}:\ n\in\mathbb{N}\right\} .\label{eq: iterated path integrals}
\end{equation}
If $\gamma$ is a stochastic process, then the map from $g$ to $Y_{T}$
is a deterministic map (independent of the sample path). Some useful
properties about $Y_{T}$ can be deduced from $S(\gamma)_{0,T}$ through
the Taylor expansion (\ref{eq:Taylor expansion}). As the simplest
example, if $g$ is well-defined almost surely, then so would $Y_{T}$,
with the exceptional set being independent of $A$. 

Motivated by the use of signature in solving differential equations,
the signature of $\gamma$ has been used to store information about
the path $\gamma$ for the purpose of e.g. handwriting recognition
(\cite{Chinese handwriting15}). Many of these signature-based methods
would benefit from a better understanding of how the signature is
related to the geometric properties of $\gamma$. For instance, certain
functionals of signature may contain more useful information about
handwriting recognition than others, in which case we may save computational
time by focusing on these features. The ``reconstruction problem''
of a path from its signature has attracted interests recently in \cite{CDNX17},
\cite{Geng17}, \cite{LX15}, \cite{LX17}, \cite{Ursitti16}. 

In the rough path literature, the first main result in this direction
was due to Hambly and Lyons \cite{HL10} that every continuous path
with bounded variation is uniquely determined by its signature up
to a \textit{tree-like equivalence}. Loosely speaking, two paths are
tree-like equivalence if one can be obtained by adding tree-like pieces
to the other, see figure below. 

\includegraphics[scale=0.6]{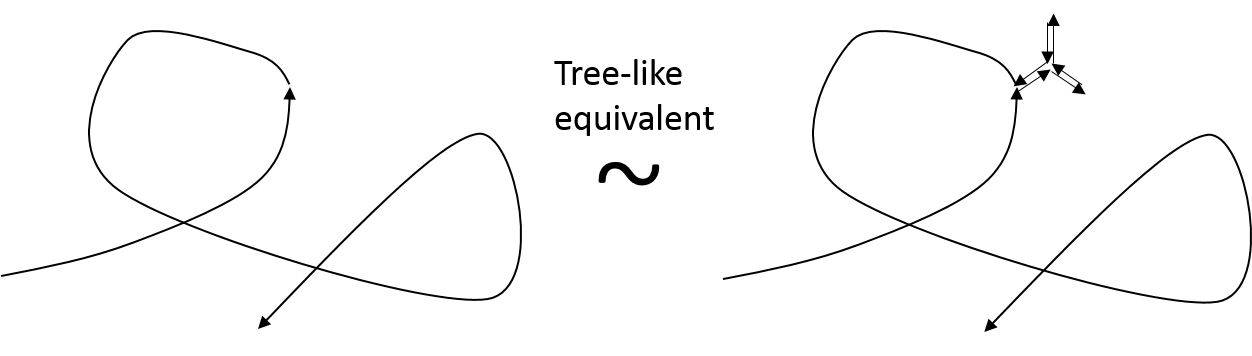}

As we do not need the precise definition of tree-like in this paper,
we refer the interested readers to \cite{HL10}. Hambly-Lyons' uniqueness
result was extended to the general rough path case in \cite{BGLY16}. 

Along with the uniqueness results mentioned above, it was also shown
(\cite{BGLY16}, \cite{HL10}) that every tree-like equivalence class
contains a unique representative path $\gamma$ which does not contain
any tree-like pieces. This representative path is called the tree-reduced
path. On the other hand, signatures have a certain algebraic structure
(see Theorem 2.15 in \cite{CTT07}) which ensures that every term
of a signature element $g$ can be recovered from looking at the tail
of $g$. Therefore, it is natural and reasonable to expect that some
intrinsic geometric properties associated with a tree-reduced rough
path can be explicitly recovered from the tail behavior of its signature. 

In the bounded variation case, it was proved that the length of a
path $\gamma$ can be recovered from the tail asymptotics of its signature
$g$ in the following way:
\begin{equation}
\|\gamma\|_{\mathrm{1-var}}=\lim_{n\rightarrow\infty}\left(n!\|g_{n}\|_{\mathrm{proj}}\right)^{\frac{1}{n}}\label{eq: the length conjecture}
\end{equation}
provided that $\gamma\in C^{1}$ when parametrized by unit speed and
the modulus of continuity $\delta_{\gamma'}$ for $\gamma'$ satisfies
$\delta_{\gamma'}(\varepsilon)=o(\varepsilon^{3/4})$ as $\varepsilon\downarrow0.$
Here $g_{n}$ is the $n$-th term of the signature $g$ and the tensor
norm is the projective norm induced by the Euclidean norm on $\mathbb{R}^{d}$
(see Definition \ref{def:projective norm}). The notation $\Vert\gamma\Vert_{1-var}$
denote the $1$-variation of $\gamma$ which is the same as the length
of $\gamma$. This formula (\ref{eq: the length conjecture}) now
also holds for general $C^{1}$ paths \cite{LX15}, piecewise linear
paths and monotonely increasing paths. Note that in dimension $1$,
the assumption that $\gamma\in C^{1}$ with respect to the unit speed
parametrization implies that $\gamma$ is monotonic in all coordinates.
Therefore (\ref{eq: the length conjecture}) is only interesting when
the dimension is greater than $1$. It has been conjectured that the
same result should hold for all tree-reduced continuous path with
bounded variation. However, very little progress has been made towards
a complete solution.

For an arbitrary continuous path with bounded variation $\gamma$,
one can easily see that 
\begin{equation}
\|g_{n}\|_{\mathrm{proj}}\leqslant\frac{\|\gamma\|_{1-\mathrm{var}}^{n}}{n!},\ \ \ \forall n\in\mathbb{N}.\label{eq:decay of BV paths}
\end{equation}
So the length conjecture (\ref{eq: the length conjecture}) is about
establishing a matching lower bound. If proved to be true in general,
it will indicate that for a tree-reduced path, the signature components
decay in an exact factorial rate. The original idea of Hambly and
Lyons for proving (\ref{eq: the length conjecture}) in the $C^{1}$-case
is looking at the lifting $X^{\lambda}$ of $\lambda\cdot\gamma$
(rescaling $\gamma$ by a large constant $\lambda$) to the hyperbolic
manifold of constant curvature $-1$ (the hyperbolic development).
It turns out that when $\lambda\rightarrow\infty,$ $X^{\lambda}$
becomes more and more like a hyperbolic geodesic in the sense that
the hyperbolic distance between the two endpoints of $X^{\lambda}$
is asymptotically comparable to its hyperbolic length. As a simple
consequence of the nature of hyperbolic development, the said hyperbolic
distance is related to the signature of $\gamma$ is a fairly explicit
way, while the hyperbolic length is the same as the original length.
In this way, one sees a lower bound for the signature in terms of
the length. It seems to us that in the deterministic setting, the
technique of hyperbolic development is essentially a $C^{1}$-technique
which requires major modification in the general bounded variation
case in quite a fundamental way.

In parallel, we could certainly ask a similar question in the rough
path context. According to Lyons \cite{Lyons98}, for a rough path
$\mathbf{X}$ with finite $p$-variation ($p\geqslant1$, see Definition
\ref{def:rough path}) the signature estimate takes the form
\begin{equation}
\|g_{n}\|_{\mathrm{proj}}\leqslant\frac{\omega(\mathbf{X})^{\frac{n}{p}}}{\left(\frac{n}{p}\right)!},\ \ \ \forall n\in\mathbb{N},\label{eq:rough factorial decay}
\end{equation}
where $\omega(\mathbf{X})$ is a constant depending on the $p$-variation
of $\mathbf{X}$ and $(\frac{n}{p})!=\Gamma(\frac{n}{p}+1)$ with
$\Gamma$ being the gamma function. To expect an analogue of (\ref{eq: the length conjecture})
for rough paths (what this actually means is not even clear at at
this point), it is natural to search lower bounds for $g_{n}$ of
the same form and look at the quantity 
\[
\tilde{L}_{p}\triangleq\limsup_{n\rightarrow\infty}\left(\left(\frac{n}{p}\right)!\|g_{n}\|_{\mathrm{proj}}\right)^{\frac{p}{n}}.
\]
On the one hand, the reason of looking at the ``limsup'' instead
of an actual limit is that, unlike the bounded variation case, the
limit does not generally exist for rough paths. For instance, one
could easily find examples of tree-reduced geometric rough paths with
infinitely many zero signature terms (for instance $\mathbf{X}_{t}\triangleq\exp(t[v,w])$
for certain vectors $v,w\in\mathbb{R}^{d}$). One might expect that
$\widetilde{L}_{p}$ is equal to the $p$-variation of the underlying
rough path. However, this cannot be the case since $\widetilde{L}_{p}=0$
for a bounded variation path when $p>1$ due to (\ref{eq:decay of BV paths}),
whereas bounded variation paths have non-zero $p$-variation. On the
other hand, if we define the ``local $p$-variation'' of a rough
path in the same way as the usual $p$-variation but additionally
by requiring that the mesh size of partitions goes to zero, it is
easy to see that the local $p$-variation of a bounded variation path
is also zero when $p>1$. Therefore, it is not entirely unreasonable
to expect that the quantity $\widetilde{L}_{p}$ recovers the local
$p$-variation of $\mathbf{X}$.

In the present article, we investigate a similar problem for the Brownian
rough path $\mathbf{B}_{t}$, which is the canonical lifting of the
Brownian motion $B_{t}$ as geometric $p$-rough paths for $2<p<3$.
One can equivalently view it as the Brownian motion coupled with the
Lévy area process. It is well-known that $B_{t}$ has a quadratic
variation process, which can be viewed as the local $2$-variation
of Brownian motion in certain probabilistic sense. In view of the
previous discussion, if we define the normalized ``limsup'' 
\begin{equation}
\widetilde{L}_{s,t}\triangleq\limsup_{n\rightarrow\infty}\left(\left(\frac{n}{2}\right)!\left\Vert \int_{s<t_{1}<\cdots<t_{n}<t}\circ\mathrm{d}B_{t_{1}}\otimes\cdots\otimes\circ\mathrm{d}B_{t_{n}}\right\Vert \right)^{\frac{2}{n}}\label{eq: definition of limsup-1}
\end{equation}
($\circ\mathrm{d}$ means the Stratonovich integral) for the Brownian
signature path under suitable tensor norms, one might expect that
$\widetilde{L}_{s,t}$ recovers some sort of quadratic variation of
the Brownian rough path. The aim of the present article is to establish
a result of this kind. Even with Lyons' estimate (\ref{eq:rough factorial decay}),
it is a priori unclear that $\widetilde{L}_{s,t}$ is even finite
since Brownian motion has infinite $2$-variation almost surely.

We are going to show that $\widetilde{L}_{s,t}$ is a deterministic
multiple of $t-s$: $\widetilde{L}_{s,t}=\kappa(t-s)$ for some deterministic
constant $\kappa$. This implies that the natural speed of Brownian
motion (i.e. its quadratic variation) can be recovered from the tail
asymptotics of its signature. In addition, we establish upper and
lower bounds on the constant $\kappa$.

On the one hand, the upper estimate is shown by using general rough
path arguments and does not reflect the tree-reduced nature of the
Brownian rough path at all. The deterministic nature of $\widetilde{L}_{s,t}$
comes from the fact that Brownian motion has independent increments.
The result holds under a wide choice of tensor norms.

On the other hand, the lower estimate is obtained by considering the
hyperbolic development of Brownian motion. Our calculation diverges
early on from the work of Hambly and Lyons \cite{HL10} for the bounded
variation paths, as we make use of martingale arguments instead of
deterministic hyperbolic analysis. Our lower estimate allows us to
conclude that the Brownian rough path is tree-reduced with probability
one and also its natural parametrization can be recovered from the
tail asymptotics of the Brownian signature. In particular, with probability
one, every Brownian rough path is uniquely determined by its signature.
This result is stronger than the existing uniqueness results for Brownian
motion in the literature (c.f. \cite{BG15}, \cite{LQ13}), since
it was only known that the signature determines the Brownian rough
path up to reparametrization.

Our main result on the upper and lower estimates of $\widetilde{L}_{s,t}$
can be summarized as follows.
\begin{thm}
\label{thm: main result}Let $B_{t}=(B_{t}^{1},\cdots,B_{t}^{d})$
be a $d$-dimensional Brownian motion ($d\geqslant2$). Define $\widetilde{L}_{s,t}$
by 

\begin{equation}
\widetilde{L}_{s,t}\triangleq\limsup_{n\rightarrow\infty}\left(\left(\frac{n}{2}\right)!\left\Vert \int_{s<t_{1}<\cdots<t_{n}<t}\circ\mathrm{d}B_{t_{1}}\otimes\cdots\otimes\circ\mathrm{d}B_{t_{n}}\right\Vert \right)^{\frac{2}{n}}\label{eq: definition of limsup}
\end{equation}
where $\Vert\cdot\Vert$ is an admissible norm (see Definition \ref{def:admissible norm}). 

(1) (upper estimate) If each element of the canonical basis $\{\mathrm{e}_{1},\cdots,\mathrm{e}_{d}\}$
of $\mathbb{R}^{d}$ has norm one with respect to $\Vert\cdot\Vert$,
then there exists a deterministic constant $\kappa_{d}\leqslant d^{2}$
depending on the choice of tensor norms, such that with probability
one 
\begin{equation}
\widetilde{L}_{s,t}=\kappa_{d}(t-s)\ \ \ \forall s<t.\label{eq:limsup QV}
\end{equation}

(2) (lower estimate) Under the $l^{p}$-norm ($1\leqslant p\leqslant2$)
on $\mathbb{R}^{d}$ and the associated projective tensor norms on
the tensor products, we have 
\[
\kappa_{d}\geqslant\frac{d-1}{2}.
\]
\end{thm}
\begin{rem}
The one dimensional case ($d=1$) is uninteresting and the result
of Theorem \ref{thm: main result} holds trivially since $\widetilde{L}_{s,t}\equiv0$
in this case.
\end{rem}
The relation between the iterated integrals and the geometry of $\gamma$
are generally complex and it is therefore somewhat surprising that
a simple formula exists relating between the quadratic variation and
length with the norms of the signature. 

There is a number of related problems to Theorem \ref{thm: main result}.
It is not known whether (\ref{eq:limsup QV}) remains true if we replace
the limsup with lim, nor do we know the exact value of $\kappa_{d}$.
Perhaps the biggest open problem of all, is how to generalise both
(\ref{eq:limsup QV}) and Hambly-Lyons' formula (\ref{eq: the length conjecture})
to rough paths, though this is known to be difficult even for (non-$C^{1}$)
bounded variation paths. We will elaborate more on the open problems
and their ramifications in Section \ref{sec:Further-remarks-and}.

Our article is organized in the following way. In Section 2, we present
some basic notions from rough path theory which are needed for our
analysis. In Section 3, we prove the first part of Theorem \ref{thm: main result}.
In Section 4, we prove the second part of Theorem \ref{thm: main result}.
In that section we also present some crucial details for understanding
the hyperbolic development which seems to be incomplete or missing
in the literature. In Section 5, we present some interesting applications
of our main result to the Brownian rough path itself. In Section 6,
we give some concluding remarks and discuss a few related further
problems.

\section{\label{sec:Notions from rough paths}Notions from rough path theory}

In this section, we present some basic notions from rough path theory
which are needed for our study. Although our main result concerns
solely about Brownian motion, key sections of our argument holds generally
for rough paths and is viewed most naturally in that context. We refer
the reader to the monographs \cite{CTT07}, \cite{FV10}, \cite{LQ02}
for a systematic introduction on rough path theory. 

Suppose that $V$ is a finite dimensional normed vector space. For
each $n\in\mathbb{N},$ define $T^{(n)}(V)\triangleq\oplus_{i=0}^{n}V^{\otimes i},$
and let $T((V))$ be the algebra of formal sequences of homogeneous
tensors $a=(a_{0},a_{1},a_{2},\cdots)$ with $a_{n}\in V^{\otimes n}$
for each $n$.
\begin{defn}
\label{def:admissible norm}A family of tensor norms $\{\|\cdot\|_{V^{\otimes n}}:\ n\geqslant1\}$
on the tensor products is called \textit{admissible }if

(1) for any $a\in V^{\otimes m}$ and $b\in V^{\otimes n},$ 
\begin{equation}
\Vert a\otimes b\Vert_{V^{\otimes(m+n)}}\leqslant\Vert a\Vert_{V^{\otimes m}}\Vert b\Vert_{V^{\otimes n}};\label{eq: admissible condition for norms}
\end{equation}

(2) for any permutation $\sigma$ on $\{1,\ldots,n\}$ and $a\in V^{\otimes n}$,
\[
\Vert\mathcal{P}^{\sigma}(a)\Vert_{V^{\otimes n}}=\Vert a\Vert_{V^{\otimes n}},
\]
where $\mathcal{P}^{\sigma}$ is the linear operator on $V^{\otimes n}$
induced by $a_{1}\otimes\cdots\otimes a_{n}\mapsto a_{\sigma(1)}\otimes\cdots\otimes a_{\sigma(n)}$
for $a_{1},\cdots,a_{n}\in V.$

We call it a family of \textit{cross-norms} if the inequality in (\ref{eq: admissible condition for norms})
is an equality. 
\end{defn}
\begin{defn}
\label{def:projective norm}The \textit{projective tensor norm} on
$V^{\otimes n}$ is defined to be 
\[
\|a\|_{\mathrm{proj}}\triangleq\inf\left\{ \sum_{l}|a_{1}^{(l)}|\cdots|a_{n}^{(l)}|:\ \mathrm{if}\ a=\sum_{l}a_{1}^{(l)}\otimes\cdots\otimes a_{n}^{(l)}\right\} .
\]
\end{defn}
It is known that the projective tensor norm is the largest cross-norm
on $V^{\otimes n}.$ In the case when $V=\mathbb{R}^{d}$ is equipped
with the $l^{1}$-norm, one can see by definition that the projective
tensor norm on $V^{\otimes n}$ is just the $l^{1}$-norm under the
canonical tensor basis induced from the one on $\mathbb{R}^{d}.$

We assume that $V$ is equipped with a family of admissible tensor
norms. Define $\triangle\triangleq\{(s,t):0\leqslant s\leqslant t\leqslant1\}$.
Given $p\geqslant1$, we denote $\lfloor p\rfloor$ as the largest
integer not exceeding $p.$
\begin{defn}
\label{def:rough path}A \textit{multiplicative functional of degree}
$n\in\mathbb{N}$ is a continuous map $\mathbf{X}_{\cdot,\cdot}=\left(1,\mathbb{X}_{\cdot,\cdot}^{1},\cdots,\mathbb{X}_{\cdot,\cdot}^{n}\right):\triangle\rightarrow T^{(n)}\left(V\right)$
which satisfies 
\[
\mathbf{X}_{s,u}\otimes\mathbf{X}_{u,t}=\mathbf{X}_{s,t},\ \mathrm{for}\ 0\leqslant s\leqslant u\leqslant t\leqslant1.
\]
Let $\mathbb{\mathbf{X}},\mathbf{Y}$ be two multiplicative functionals
of degree $n.$ Define 
\[
d_{p}\left(\mathbf{X},\mathbf{Y}\right)\triangleq\max_{1\leqslant i\leqslant n}\sup_{\mathcal{P}}\left(\sum_{l}\left\Vert \mathbb{X}_{t_{l-1},t_{l}}^{i}-\mathbb{Y}_{t_{l-1},t_{l}}^{i}\right\Vert _{V^{\otimes i}}^{\frac{p}{i}}\right)^{\frac{i}{p}},
\]
where the supremum is taken over all possible finite partitions $\mathcal{P}=(t_{0}<t_{1}<\ldots<t_{n})$
of $[0,1]$. $d_{p}$ is called the \textit{p-variation metric}. If
$d_{p}\left(\mathbf{X},\mathbf{1}\right)<\infty$ where $\mathbf{1}=(1,0,\cdots,0)$,
we say that $\mathbf{X}$ has \textit{finite p-variati}on. A multiplicative
functional of degree $\lfloor p\rfloor$ with finite $p$-variation
is called a \textit{p-rough path}. 
\end{defn}
The following important result, proved by Lyons \cite{Lyons98}, asserts
that ``iterated path integrals'' for a rough path are also well
defined.
\begin{thm}
{[}Lyons' extension theorem{]}\label{thm: Lyons' extension theorem}
Let $\mathbf{X}=(1,\mathbb{X}^{1},\cdots,\mathbb{X}^{\lfloor p\rfloor})$
be a $p$-rough path. Then for any $n\geqslant\lfloor p\rfloor+1,$
there exists a unique continuous map $\mathbb{X}^{n}:\ \Delta\rightarrow V^{\otimes n}$,
such that 
\[
\mathbb{X}_{\cdot,\cdot}\triangleq\left(1,\mathbb{X}{}_{\cdot,\cdot}^{1},\cdots,\mathbb{X}_{\cdot,\cdot}^{\lfloor p\rfloor},\cdots,\mathbb{X}_{\cdot,\cdot}^{n},\cdots\right)
\]
is a multiplicative functional in $T\left((V)\right)$ whose projection
onto $T^{(n)}(V)$ has finite $p$-variation for every $n.$
\end{thm}
\begin{rem}
Due to the multiplicative structure, when we consider a rough path,
one could simply look at the path $t\mapsto\mathbf{X}_{0,t}$ whose
increments are defined to be $\mathbf{X}_{s}^{-1}\otimes\mathbf{X}_{t}.$ 
\end{rem}
\medskip{}
\begin{rem}
When $p=1$ and $\mathbf{X}$ is a continuous path with bounded variation,
all the previous notions reduces to the classical iterated path integrals
defined in the sense of Lebesgue-Stieltjes.
\end{rem}
\begin{defn}
Let $\mathbf{X}$ be a $p$-rough path. The path $t\mapsto\mathbb{X}_{0,t}\in T((V))$
defined by Lyons' extension theorem is called the\textit{ signature
path }of $\mathbf{X}.$ The quantity $\mathbb{X}_{0,1}$ is called
the \textit{signature} of $\mathbf{X}.$ 
\end{defn}
Among general rough paths there is a fundamental class of paths called
geometric rough paths.
\begin{defn}
For a continuous path with bounded variation $\gamma:\ [0,1]\rightarrow V$,
define 
\[
\mathbb{X}_{s,t}^{n}=\int_{s<u_{1}<\cdots<u_{n}<t}\mathrm{d}\gamma_{u_{1}}\otimes\cdots\otimes\mathrm{d}\gamma_{u_{n}},\ \ \ n\geqslant1,\ s\leqslant t.
\]
The closure of the space 
\[
\{(1,\mathbb{X}_{s,t}^{1},\ldots,\mathbb{X}_{s,t}^{\lfloor p\rfloor}):\ \gamma\mbox{ is a continuous path with bounded variation}\}
\]
under the $p$-variation metric $d_{p}$ is called the space of \textit{geometric
$p$-rough paths}. 
\end{defn}
\begin{rem}
According to \cite{CTT07}, Theorem 2.15, the iterated integrals of
bounded variation path $\gamma$ in $\mathbb{R}^{d}$ satisfies 
\begin{align}
 & \int_{s<u_{1}<\cdots<u_{n}<t}\mathrm{d}\gamma_{u_{1}}^{i_{1}}\cdots\mathrm{d}\gamma_{u_{n}}^{i_{n}}\cdot\int_{s<u_{n+1}<\cdots<u_{n+k}<t}\mathrm{d}\gamma_{u_{n+1}}^{i_{n+1}}\cdots\mathrm{d}\gamma_{u_{n+k}}^{i_{n+k}}\nonumber \\
= & \sum_{\sigma\in\mathcal{S}(n,k)}\int_{s<u_{1}<\cdots<u_{n+k}<t}\mathrm{d}\gamma_{u_{1}}^{i_{\sigma^{-1}(1)}}\cdots\mathrm{d}\gamma_{u_{n+k}}^{i_{\sigma^{-1}(n+k)}},\label{eq:shuffle product formula-1}
\end{align}
where $\cdot$ is real number multiplication and $\mathcal{S}(n,k)$
contains all permutations $\sigma:\{1,\ldots,n+k\}\rightarrow\{1,\ldots,n+k\}$
such that
\[
\sigma(1)<\ldots<\sigma(n);\;\sigma(n+1)<\ldots<\sigma(n+k).
\]
In other words, the product of $n$-th and $k$-th order iteratred
integrals can be rewritten as a linear combination of $n+k$-th order
iterated integrals. The property (\ref{eq:shuffle product formula-1})
extends to geometric rough paths. An equivalent, but coordinate invariant,
formulation of (\ref{eq:shuffle product formula-1}) is that 
\begin{equation}
\mathbb{X}_{s,t}^{n}\otimes\mathbb{X}_{s,t}^{k}=\sum_{\sigma\in S(n,k)}\mathcal{P}^{\sigma}(\mathbb{X}_{s,t}^{n+k}),\label{eq:Shuffle version 2}
\end{equation}
where $\mathbb{X}^{k}$ is the $n$-th term of the signature (see
Theorem \ref{thm: Lyons' extension theorem}), and $\mathcal{P}^{\sigma}$
is the permutation of tensors map defined in (2), Definition \ref{def:admissible norm}. 

In fact, (\ref{eq:Shuffle version 2}) and the multiplicative property
\begin{equation}
\sum_{k=0}^{n}\mathbb{X}_{s,u}^{k}\otimes\mathbb{X}_{u,t}^{n-k}=\mathbb{X}_{s,t}^{n}\label{eq:multiplicative property}
\end{equation}
are the two most fundamental algebraic properties of iterated integrals. 
\end{rem}
The space of geometric rough paths plays a fundamental role in rough
path theory. In particular, a complete integration and differential
equation theory with respect to geometric rough paths has been established
by Lyons \cite{Lyons98}. The rough path theory has significant applications
in probability theory, mainly due to the fact that a wide class of
interesting stochastic processes can be regarded as geometric rough
paths in a canonical way in the sense of natural approximations.

In particular, it is known that (c.f. \cite{Sipilainen93}) a multidimensional
Brownian motion $B_{t}$ admits a canonical lifting as geometric $p$-rough
path $\mathbf{B}_{t}$ with $p\in(2,3)$. $\mathbf{B}_{t}$ is called
the \textit{Brownian rough path.} The corresponding \textit{Brownian
signature path}, determined by Lyons' extension theorem, is denoted
as 
\[
\mathbb{B}_{s,t}=(1,\mathbb{B}_{s,t}^{1},\mathbb{B}_{s,t}^{2},\cdots),\ \ \ s\leqslant t.
\]
Under the canonical tensor basis on tensor products over $V\triangleq\mathbb{R}^{d}$,
for each word $(i_{1},\cdots,i_{n})$ over $\{1,\cdots,d\}$, the
coefficient of $\mathbb{B}_{s,t}^{n}$ with respect to $\mathrm{e}_{i_{1}}\otimes\cdots\otimes\mathrm{e}_{i_{n}}$
coincides with the iterated Stratonovich integral (c.f. \cite{CTT07}):
\[
\mathbb{B}_{s,t}^{n;i_{1},\cdots,i_{n}}=\int_{s<u_{1}<\cdots<u_{n}<t}\circ\mathrm{d}B_{u_{1}}^{i_{1}}\cdots\circ\mathrm{d}B_{u_{n}}^{i_{n}}.
\]

For a given family of admissible tensor norms, we define 
\[
\widetilde{L}_{s,t}\triangleq\limsup_{n\rightarrow\infty}\left(\left(\frac{n}{2}\right)!\left\Vert \mathbb{B}_{s,t}^{n}\right\Vert \right)^{\frac{2}{n}},\ \ \ s\leqslant t.
\]
Lyons \cite{Lyons98} established a uniform bound for the $n$-term
in the signature, $\mathbb{X}_{s,t}^{n}$, for general $p$-rough
path as 
\begin{equation}
\left\Vert \mathbb{X}_{s,t}^{n}\right\Vert \leq\frac{\omega(\mathbf{X})^{n}}{(\frac{n}{p})!},\label{eq:Lyons estimate}
\end{equation}
where $\omega(\mathbf{X})$ depends on the $p$-variation of $\mathbf{X}$.
Since Brownian motion is a $p$-rough path for all $p>2$ but not
$p=2$, the finiteness of $\tilde{L}_{s,t}$ does not follow from
Lyons estimate (\ref{eq:Lyons estimate}). In section \ref{sec:upper bound},
we will establish yet another bound of iterated integrals, which is
sharper than those available in the literature (e.g. \cite{KloedenPlaten02,BenArous89})
by a geometric factor, and holds not just for $\tilde{L}_{s.t}$ but
also $\sup_{s\leq u\leq v\leq t}\tilde{L}_{u,v}$, which is essential
in dealing with null sets later on. 

\section{\label{sec:upper bound}First part of the main result: the upper
estimate }

In this section, we develop the proof of the first part of Theorem
\ref{thm: main result}. 
\begin{lem}
\label{lem: estimating the sup-L^1 norm for the Brownian signature}The
signature coefficients $\mathbb{B}_{s,t}^{n;i_{1},\cdots,i_{n}}$
satisfy the following estimate: 
\[
\mathbb{E}\left[\sup_{s\leqslant u\leqslant t}\left|\mathbb{B}_{s,u}^{n;i_{1},\cdots,i_{n}}\right|\right]\leqslant\left(\frac{1}{2}+\sqrt{2}\right)\left(\frac{\mathrm{e}}{\sqrt{2}\pi}\right)^{\frac{1}{2}}\frac{2^{\frac{n}{2}}}{(n-2)^{\frac{1}{4}}\sqrt{n!}}(t-s)^{\frac{n}{2}}
\]
for all $s<t,$ $n\geqslant1$ and $1\leqslant i_{1},\cdots,i_{n}\leqslant d.$ 
\end{lem}
\begin{proof}
By translation, it suffices to consider the case when $s=0.$

We first estimate the second moment of $\mathbb{B}_{0,u}^{n;i_{1},\cdots,i_{n}}.$
Recall from (\ref{eq:shuffle product formula-1}) the shuffle product
formula
\begin{align}
\mathbb{B}_{s,t}^{n;i_{1},\cdots,i_{n}}\cdot\mathbb{B}_{s,t}^{k;i_{n+1},\cdots,i_{n+k}}= & \sum_{\sigma\in\mathcal{S}(n,k)}\mathbb{B}_{s,t}^{n+k;i_{\sigma^{-1}(1)},\ldots,i_{\sigma^{-1}(n+k)}},\label{eq:shuffle product formula}
\end{align}
where $\cdot$ is real number multiplication and $\mathcal{S}(n,k)$
denotes the set of all permutations on $\{1,\ldots,n+k\}$ such that
\[
\sigma(1)<\ldots<\sigma(n);\;\sigma(n+1)<\ldots<\sigma(n+k).
\]
Applying this formula (\ref{eq:shuffle product formula}) and taking
expectation, we have
\[
\mathbb{E}\left[\left|\mathbb{B}_{0,u}^{n;i_{1},\cdots,i_{n}}\right|^{2}\right]=\sum_{\sigma\in\mathcal{S}(n,n)}\mathbb{E}\left[\mathbb{B}_{0,u}^{2n;j_{\sigma^{-1}(1)},\cdots,j_{\sigma^{-1}(2n)}}\right],
\]
where $(j_{1},\cdots,j_{2n})\triangleq(i_{1},\cdots,i_{n},i_{1},\cdots,i_{n})$.
Since $|\mathcal{S}(n,n)|=\frac{(2n)!}{(n!)^{2}}$, we have 
\[
\mathbb{E}\left[\left|\mathbb{B}_{0,u}^{n;i_{1},\cdots,i_{n}}\right|^{2}\right]\leq\frac{(2n)!}{(n!)^{2}}\max_{\sigma\in\mathcal{S}(n,n)}\left|\mathbb{E}\left[\mathbb{B}_{0,u}^{2n;j_{\sigma^{-1}(1)},\cdots,j_{\sigma^{-1}(2n)}}\right]\right|
\]
On the other hand, we know from \cite{Fawcett03} (see also Prop 4.10
in \cite{LyonsVictoir04}) that 
\[
\mathbb{E}\left[\mathbb{B}_{0,u}^{2n}\right]=\frac{u^{n}}{n!2^{n}}\left(\sum_{i=1}^{d}\mathrm{e}_{i}\otimes\mathrm{e}_{i}\right)^{\otimes n}.
\]
In particular, every coefficient of basis elements $\mathrm{e}_{k_{1}}\otimes\ldots\otimes\mathrm{e}_{k_{2n}}$
in $\mathbb{E}\left[\mathbb{B}_{0,u}^{2n}\right]$ is either zero
or $\frac{u^{n}}{n!2^{n}}.$ Therefore, 
\begin{align}
\mathbb{E}\left[\left|\mathbb{B}_{0,u}^{n;i_{1},\cdots,i_{n}}\right|^{2}\right] & \leqslant\frac{(2n)!}{(n!)^{2}}\cdot\frac{u^{n}}{n!2^{n}}\nonumber \\
 & \leqslant\frac{\mathrm{e}(2n)^{2n+\frac{1}{2}}\mathrm{e}^{-2n}}{2\pi n^{2n+1}\mathrm{e}^{-2n}}\cdot\frac{u^{n}}{n!2^{n}}\quad\text{(by Stirling's approximation)}\nonumber \\
 & =\frac{\mathrm{e}}{\sqrt{2}\pi}\frac{2^{n}}{\sqrt{n}n!}u^{n}.\label{eq: second moment estimate for the Brownian signature}
\end{align}
Secondly, by the definition of iterated integral, 
\[
\mathbb{B}_{0,u}^{n;i_{1},\cdots,i_{n}}=\int_{0}^{u}\mathbb{B}_{0,t}^{n-1;i_{1},\cdots,i_{n-1}}\circ\mathrm{d}B_{t}^{i_{n}}
\]
and hence
\begin{align}
d\mathbb{B}_{0,u}^{n;i_{1},\cdots,i_{n}} & =\mathbb{B}_{0,u}^{n-1;i_{1},\cdots,i_{n-1}}\circ\mathrm{d}B_{u}^{i_{n}}\nonumber \\
 & =\mathbb{B}_{0,u}^{n-1;i_{1},\cdots,i_{n-1}}\cdot\mathrm{d}B_{u}^{i_{n}}+\frac{1}{2}d\mathbb{B}_{0,u}^{n-1;i_{1},\cdots,i_{n-1}}\cdot\mathrm{d}B_{u}^{i_{n}}\;\text{(Ito to Stratonovich)}\nonumber \\
 & =\mathbb{B}_{0,u}^{n-1;i_{1},\cdots,i_{n-1}}\cdot\mathrm{d}B_{u}^{i_{n}}+\frac{1}{2}\left(\mathbb{B}_{0,u}^{n-2;i_{1},\cdots,i_{n-2}}\circ dB_{u}^{i_{n-1}}\right)\cdot\mathrm{d}B_{u}^{i_{n}}\label{eq:SDE signature}\\
 & =\mathbb{B}_{0,u}^{n-1;i_{1},\cdots,i_{n-1}}\cdot\mathrm{d}B_{u}^{i_{n}}+\frac{1}{2}\delta_{i_{n-1},i_{n}}\mathbb{B}_{0,u}^{n-2;i_{1},\cdots,i_{n-2}}\mathrm{d}u.\label{SDE signature 2}
\end{align}
By integrating and taking supremum, 
\begin{align*}
 & \mathbb{E}\left[\sup_{0\leqslant u\leqslant t}\left|\mathbb{B}_{0,u}^{n;i_{1},\cdots,i_{n}}\right|\right]\\
 & \leqslant\mathbb{E}\left[\sup_{0\leqslant u\leqslant t}\left|\int_{0}^{u}\mathbb{B}_{0,v}^{n-1;i_{1},\cdots,i_{n-1}}\cdot\mathrm{d}B_{v}^{i_{n}}\right|\right]+\sup_{0\leqslant u\leqslant t}\frac{1}{2}\int_{0}^{u}\mathbb{E}\left[\left|\mathbb{B}_{0,v}^{n-2;i_{1},\cdots,i_{n-2}}\right|\right]\mathrm{d}v\\
 & \leqslant\mathbb{E}\left[\sup_{0\leqslant u\leqslant t}\left|\int_{0}^{u}\mathbb{B}_{0,v}^{n-1;i_{1},\cdots,i_{n-1}}\cdot\mathrm{d}B_{v}^{i_{n}}\right|\right]+\frac{1}{2}\int_{0}^{t}\sqrt{\mathbb{E}\left[\left|\mathbb{B}_{0,v}^{n-2;i_{1},\cdots,i_{n-2}}\right|^{2}\right]}\mathrm{d}v.
\end{align*}
It follows from (\ref{eq: second moment estimate for the Brownian signature})
that 
\begin{align*}
 & \mathbb{E}\left[\sup_{0\leqslant u\leqslant t}\left|\mathbb{B}_{0,u}^{n;i_{1},\cdots,i_{n}}\right|\right]\\
 & \leqslant\mathbb{E}\left[\sup_{0\leqslant u\leqslant t}\left|\int_{0}^{u}\mathbb{B}_{0,v}^{n-1;i_{1},\cdots,i_{n-1}}\cdot\mathrm{d}B_{v}^{i_{n}}\right|\right]+\frac{1}{2}\int_{0}^{t}\left(\frac{\mathrm{e}}{\sqrt{2}\pi}\frac{2^{n-2}}{\sqrt{n-2}(n-2)!}v^{n-2}\right)^{\frac{1}{2}}\mathrm{d}v\\
 & =\mathbb{E}\left[\sup_{0\leqslant u\leqslant t}\left|\int_{0}^{u}\mathbb{B}_{0,v}^{n-1;i_{1},\cdots,i_{n-1}}\cdot\mathrm{d}B_{v}^{i_{n}}\right|\right]+\frac{1}{2}\left(\frac{\mathrm{e}}{\sqrt{2}\pi}\right)^{\frac{1}{2}}\frac{2^{\frac{n}{2}}}{(n-2)^{\frac{1}{4}}\sqrt{(n-2)!}n}t^{\frac{n}{2}}\\
 & \leqslant\mathbb{E}\left[\sup_{0\leqslant u\leqslant t}\left|\int_{0}^{u}\mathbb{B}_{0,v}^{n-1;i_{1},\cdots,i_{n-1}}\cdot\mathrm{d}B_{v}^{i_{n}}\right|\right]+\frac{1}{2}\left(\frac{\mathrm{e}}{\sqrt{2}\pi}\right)^{\frac{1}{2}}\frac{2^{\frac{n}{2}}}{(n-2)^{\frac{1}{4}}\sqrt{n!}}t^{\frac{n}{2}}.
\end{align*}
The first term can be estimated easily by using Doob's $L^{p}$-inequality:
\begin{align}
\mathbb{E}\left[\sup_{0\leqslant u\leqslant t}\left|\int_{0}^{u}\mathbb{B}_{0,v}^{n-1;i_{1},\cdots,i_{n-1}}\cdot\mathrm{d}B_{v}^{i_{n}}\right|\right] & \leqslant\left\Vert \sup_{0\leqslant u\leqslant t}\left|\int_{0}^{u}\mathbb{B}_{0,v}^{n-1;i_{1},\cdots,i_{n-1}}\cdot\mathrm{d}B_{v}^{i_{n}}\right|\right\Vert _{2}\nonumber \\
 & \leqslant2\left\Vert \int_{0}^{t}\mathbb{B}_{0,v}^{n-1;i_{1},\cdots,i_{n-1}}\cdot\mathrm{d}B_{v}^{i_{n}}\right\Vert _{2}\nonumber \\
 & =2\left(\int_{0}^{t}\mathbb{E}\left[\left|\mathbb{B}_{0,v}^{n-1;i_{1},\cdots,i_{n-1}}\right|^{2}\right]\mathrm{d}v\right)^{\frac{1}{2}}\nonumber \\
 & \leqslant2\left(\int_{0}^{t}\frac{\mathrm{e}}{\sqrt{2}\pi}\frac{2^{n-1}}{\sqrt{n-1}(n-1)!}v^{n-1}\mathrm{d}v\right)^{\frac{1}{2}}\label{eq:supremem of ito integral}\\
 & =\sqrt{2}\left(\frac{\mathrm{e}}{\sqrt{2}\pi}\right)^{\frac{1}{2}}\frac{2^{\frac{n}{2}}}{(n-2)^{\frac{1}{4}}\sqrt{n!}}t^{\frac{n}{2}}.\nonumber 
\end{align}

Now the desired estimate follows immediately. 
\end{proof}
\begin{rem}
Second moment estimate on iterated Stratonovich's integrals was studied
by Ben Arous \cite{BenArous89} through iterated Itô's integrals.
Here the estimate (\ref{eq: second moment estimate for the Brownian signature})
we obtained through the shuffle product formula and the Brownian expected
signature is sharper in the exponential factor. 
\end{rem}
Now we are able to establish the following main upper estimate.
\begin{prop}
\label{prop: upper bound}Suppose that the tensor products $(\mathbb{R}^{d})^{\otimes n}$
are equipped with given admissible norms, under which each element
of the standard basis $\{\mathrm{e}_{1},\cdots,\mathrm{e}_{d}\}$
of $\mathbb{R}^{d}$ has norm one. Then for each $s<t,$ with probability
one, we have 
\[
\max\left\{ \limsup_{n\rightarrow\infty}\left(\left(\frac{n}{2}\right)!\sup_{s\leqslant u\leqslant t}\left\Vert \mathbb{B}_{s,u}^{n}\right\Vert \right)^{\frac{2}{n}},\limsup_{n\rightarrow\infty}\left(\left(\frac{n}{2}\right)!\sup_{s\leqslant u\leqslant t}\left\Vert \mathbb{B}_{u,t}^{n}\right\Vert \right)^{\frac{2}{n}}\right\} \leqslant d^{2}(t-s).
\]
\end{prop}
\begin{proof}
Since the tensor norms are admissible, for each multi-index, 
\[
\Vert\mathrm{e}_{i_{1}}\otimes\cdots\otimes\mathrm{e}_{i_{n}}\Vert\leq\Vert\mathrm{e}_{i_{1}}\Vert\Vert\mathrm{e}_{i_{2}}\Vert\ldots\Vert\mathrm{e}_{i_{n}}\Vert=1.
\]
This together with the triangle inequality implies that 
\begin{align*}
\left\Vert \mathbb{B}_{s,u}^{n}\right\Vert  & =\left\Vert \sum_{i_{1},\cdots,i_{n}=1}^{d}\mathbb{B}_{s,u}^{n;i_{1},\cdots,i_{n}}\mathrm{e}_{i_{1}}\otimes\cdots\otimes\mathrm{e}_{i_{n}}\right\Vert \\
 & \leqslant\sum_{i_{1},\cdots,i_{n}=1}^{d}\left|\mathbb{B}_{s,u}^{n;i_{1},\cdots,i_{n}}\right|,
\end{align*}
and thus 
\begin{equation}
\mathbb{E}\left[\sup_{s\leqslant u\leqslant t}\|\mathbb{B}_{s,t}^{n}\|\right]\leqslant\sum_{i_{1},\cdots,i_{n}=1}^{d}\mathbb{E}\left[\sup_{s\leqslant u\leqslant t}\left|\mathbb{B}_{s,u}^{n;i_{1},\cdots,i_{n}}\right|\right].\label{eq:after triangle}
\end{equation}
The summand on the right hand side of (\ref{eq:after triangle}) can
be bounded by Lemma \ref{lem: estimating the sup-L^1 norm for the Brownian signature},
and we arrive at 
\[
\mathbb{E}\left[\sup_{s\leqslant u\leqslant t}\|\mathbb{B}_{s,u}^{n}\|\right]\leqslant d^{n}\cdot\frac{C2^{\frac{n}{2}}}{(n-2)^{\frac{1}{4}}\sqrt{n!}}(t-s)^{\frac{n}{2}},
\]
where 
\[
C\triangleq\left(\frac{1}{2}+\sqrt{2}\right)\left(\frac{\mathrm{e}}{\sqrt{2}\pi}\right)^{\frac{1}{2}}.
\]

Now for each $r>(t-s),$ we have 
\[
\mathbb{P}\left(\sup_{s\leqslant u\leqslant t}\|\mathbb{B}_{s,u}^{n}\|>\frac{Cd^{n}2^{\frac{n}{2}}}{(n-2)^{\frac{1}{4}}\sqrt{n!}}r^{\frac{n}{2}}\right)\leqslant\left(\frac{t-s}{r}\right)^{\frac{n}{2}}.
\]
By the Borel-Cantelli lemma, with probability one (with null set depending
on $s$ and $t$), 
\[
\sup_{s\leqslant u\leqslant t}\|\mathbb{B}_{s,u}^{n}\|\leqslant\frac{Cd^{n}2^{\frac{n}{2}}}{(n-2)^{\frac{1}{4}}\sqrt{n!}}r^{\frac{n}{2}}
\]
for all sufficiently large $n.$ It follows from Stirling's approximation
that with probability one, 
\[
\limsup_{n\rightarrow\infty}\left(\left(\frac{n}{2}\right)!\sup_{s\leqslant u\leqslant t}\left\Vert \mathbb{B}_{s,u}^{n}\right\Vert \right)^{\frac{2}{n}}\leqslant\lim_{n\rightarrow\infty}\left(\left(\frac{n}{2}\right)!\frac{Cd^{n}2^{\frac{n}{2}}}{(n-2)^{\frac{1}{4}}\sqrt{n!}}r^{\frac{n}{2}}\right)^{\frac{2}{n}}=d^{2}r.
\]
By taking a rational sequence $r\downarrow(t-s)$, we conclude that
with probability one (with null set depending on $t$ and $s$), 
\begin{equation}
\limsup_{n\rightarrow\infty}\left(\left(\frac{n}{2}\right)!\sup_{s\leqslant u\leqslant t}\left\Vert \mathbb{B}_{s,u}^{n}\right\Vert \right)^{\frac{2}{n}}\leqslant d^{2}(t-s).\label{eq: one-sided supremum estimate}
\end{equation}

Next we will bound $\sup_{v\leq u\leq t}\Vert\mathbb{B}_{u,t}^{n}\Vert$
using the reversability of Brownian motion. For the estimate involving
$\mathbb{B}_{u,t}^{n},$ observe that 
\begin{align*}
\mathbb{B}_{u,t}^{n} & =\int_{u<v_{1}<\cdots<v_{n}<t}\mathrm{d}B_{v_{1}}\otimes\cdots\otimes\mathrm{d}B_{v_{n}}\\
 & =\int_{0<r_{n}<\cdots<r_{1}<t-u}\mathrm{d}B_{t-r_{1}}\otimes\cdots\otimes\mathrm{d}B_{t-r_{n}}\\
 & =\mathcal{P}^{\tau}\left(\int_{0<r_{1}<\cdots<r_{n}<t-u}\mathrm{d}W_{r_{1}}\otimes\cdots\otimes\mathrm{d}W_{r_{n}}\right),
\end{align*}
where $W_{r}\triangleq B_{t-r}-B_{t}$ $(0\leqslant r\leqslant t-u)$
is again a Brownian motion and $\mathcal{P}^{\tau}$ is the linear
transformation on $\left(\mathbb{R}^{d}\right)^{\otimes n}$ determined
by $\xi_{1}\otimes\cdots\otimes\xi_{n}\mapsto\xi_{n}\otimes\cdots\otimes\xi_{1}.$
It follows that 
\[
\|\mathbb{B}_{u,t}^{n}\|=\left\Vert \mathbb{W}_{0,t-u}^{n}\right\Vert .
\]
Therefore, 
\[
\sup_{s\leqslant u\leqslant t}\|\mathbb{B}_{u,t}^{n}\|=\sup_{0\leqslant v\leqslant t-s}\|\mathbb{W}_{0,v}^{n}\|.
\]
Therefore, what we have proven before shows that 
\begin{align*}
\limsup_{n\rightarrow\infty}\left(\left(\frac{n}{2}\right)!\sup_{s\leqslant u\leqslant t}\left\Vert \mathbb{B}_{u,t}^{n}\right\Vert \right)^{\frac{2}{n}} & =\limsup_{n\rightarrow\infty}\left(\left(\frac{n}{2}\right)!\sup_{0\leqslant v\leqslant t-s}\left\Vert \mathbb{W}_{0,v}^{n}\right\Vert \right)^{\frac{2}{n}}\\
 & \leqslant d^{2}(t-s)
\end{align*}
for almost surely. 
\end{proof}
Recall that $\widetilde{L}_{s,t}$ is defined by
\[
\widetilde{L}_{s,t}\triangleq\limsup_{n\rightarrow\infty}\left(\left(\frac{n}{2}\right)!\left\Vert \int_{s<t_{1}<\cdots<t_{n}<t}\circ\mathrm{d}B_{t_{1}}\otimes\cdots\otimes\circ\mathrm{d}B_{t_{n}}\right\Vert \right)^{\frac{2}{n}}
\]
under given admissible tensor norms. It is immediate from Proposition
\ref{prop: upper bound} that $\widetilde{L}_{s,t}\leqslant d^{2}(t-s)$
for almost surely.

Now we are going to show that $\widetilde{L}_{s,t}$ is almost surely
a deterministic constant.

Recall that $g\in T((\mathbb{R}^{d}))$ is a group-like element if
and only if $g=(1,g^{1},g^{2},\cdots)$ satisfies 
\[
g^{n}\otimes g^{k}=\sum_{\sigma\in\mathcal{S}(n,k)}\mathcal{P}^{\sigma}\big(g^{n+k}\big)
\]
where, as mentioned, $\mathcal{P}^{\sigma}$ is the unique linear
map on $\big(\mathbb{R}^{d}\big)^{\otimes(n+k)}$ such that 
\[
\mathcal{P}^{\sigma}(v_{1}\otimes\ldots\otimes v_{n+k})=v_{\sigma(1)}\otimes\ldots\otimes v_{\sigma(n+k)}
\]
and $\mathcal{S}(n,k)$ denotes the set of permutations on $n+k$
elements such that 
\[
\sigma(1)<\sigma(2)<\ldots<\sigma(n),\;\sigma(n+1)<\ldots<\sigma(n+k).
\]
This is in fact equivalent to the shuffle product formula (\ref{eq:shuffle product formula})
mentioned earlier. In particular, the signature of a geometric rough
path is always a group-like element. 
\begin{lem}
\label{lem: infinitely many non-zero components for a nontrivial group-like element}Let
$g=(1,g^{1},g^{2},\cdots)$ be a non-trivial group-like element in
the tensor algebra $T((\mathbb{R}^{d}))$, where the tensor products
are equipped with given admissible norms. Then $g$ has infinitely
many non-zero components. 
\end{lem}
\begin{proof}
Suppose that $g^{k}\neq0$ for some $k\geqslant1.$ According to the
shuffle product formula, for each $n\geqslant1,$ 
\[
\left(g^{k}\right)^{\otimes n}=\sum_{\sigma\in\mathcal{S}(k,\cdots,k)}\mathcal{P}^{\sigma}\left(g^{nk}\right),
\]
with $\mathcal{S}(k,\ldots,k)$ denoting the set of permutations on
$nk$ elements such that 
\begin{align*}
\sigma(1) & <\ldots<\sigma(k);\\
\sigma(k+1) & <\ldots<\sigma(2k);\\
 & \vdots\\
\sigma((n-1)k+1) & <\ldots<\sigma(nk).
\end{align*}
 Since the tensor norms are admissible, we have 
\begin{align*}
\|g^{k}\|^{n} & \leqslant\sum_{\sigma\in\mathcal{S}(k,\cdots,k)}\left\Vert \mathcal{P}^{\sigma}\left(g^{nk}\right)\right\Vert \\
 & =\frac{(nk)!}{(k!)^{n}}\|g^{nk}\|.
\end{align*}
In particular, $g^{nk}\neq0$ for all $n.$ 
\end{proof}
\begin{lem}
\label{lem: factorial lemma} Given $\alpha>0$, there exists a constant
$C>0$, such that 
\[
\frac{\left(\frac{n}{p}\right)!}{\left(\frac{n-\alpha}{p}\right)!}\leqslant Cn^{\frac{\alpha}{p}},\ \ \ \forall n>2\alpha,\ p\geqslant1.
\]
\end{lem}
\begin{proof}
According to Stirling's approximation, there exist constants $C_{1},C_{2}>0,$
such that 
\[
C_{1}\lambda^{\lambda+\frac{1}{2}}\mathrm{e}^{-\lambda}\leqslant\lambda!\leqslant C_{2}\lambda^{\lambda+\frac{1}{2}}\mathrm{e}^{-\lambda},\ \ \ \forall\lambda>0.
\]
Therefore, 
\begin{align*}
\frac{\left(\frac{n}{p}\right)!}{\left(\frac{n-\alpha}{p}\right)!} & \leqslant\frac{C_{2}\left(\frac{n}{p}\right)^{\frac{n}{p}+\frac{1}{2}}\mathrm{e}^{-\frac{n}{p}}}{C_{1}\left(\frac{n-\alpha}{p}\right)^{\frac{n-\alpha}{p}+\frac{1}{2}}\mathrm{e}^{-\frac{n-\alpha}{p}}}\\
 & =\frac{C_{2}}{C_{1}\left(p\mathrm{e}\right)^{\frac{\alpha}{p}}}\left(1+\frac{\alpha}{n-\alpha}\right)^{\frac{n-\alpha}{p}+\frac{1}{2}}n^{\frac{\alpha}{p}}\\
 & \leqslant\frac{\sqrt{2}C_{2}\mathrm{e}^{\alpha}}{C_{1}}n^{\frac{\alpha}{p}}.
\end{align*}
Choosing $C\triangleq\sqrt{2}C_{2}\mathrm{e}^{\alpha}/C_{1}$ suffices. 
\end{proof}
The following deterministic sub-additivity property is essential for
us.
\begin{prop}
\label{prop: sub-additivity}(Subadditivity estimate) Suppose that
$\mathbf{X}$ is a rough path, where the tensor products are equipped
with given admissible norms. Let $\mathbb{X}_{s,t}^{n}$ be the degree-$n$
iterated integrals of $\mathbf{X}$ on the interval $[s,t]$, as defined
by Theorem \ref{thm: Lyons' extension theorem}. Let $p\geqslant1$
be a given constant. Define 
\[
\widetilde{l}_{s,t}\triangleq\limsup_{n\rightarrow\infty}\left\Vert \left(\frac{n}{p}\right)!\mathbb{X}_{s,t}^{n}\right\Vert ^{\frac{p}{n}},\ \ s\leqslant t.
\]
Then $(s,t)\mapsto\widetilde{l}_{s,t}$ is sub-additive, i.e. 
\[
\widetilde{l}_{s,t}\leqslant\widetilde{l}_{s,u}+\widetilde{l}_{u,t}
\]
for $s\leqslant u\leqslant t.$ 
\end{prop}
\begin{proof}
We may assume that $\widetilde{l}_{s,u},\widetilde{l}_{u,t}$ are
both finite. Moreover, we may also assume that both of $\mathbb{X}_{s,u}$
and $\mathbb{X}_{u,t}$ are non-trivial, otherwise the desired inequality
is trivial due to the multiplicative property (\ref{eq:multiplicative property}).
From Lemma \ref{lem: infinitely many non-zero components for a nontrivial group-like element},
$\mathbb{X}_{s,u}$ and $\mathbb{X}_{u,t}$ have infinitely many non-zero
components.

Given integers $\alpha>2p$ and $n>2\alpha,$ according to the multiplicative
property (\ref{eq:multiplicative property}), we have 
\begin{align*}
\left\Vert \mathbb{X}_{s,t}^{n}\right\Vert  & =\left\Vert \sum_{k=0}^{n}\mathbb{X}_{s,u}^{k}\otimes\mathbb{X}_{u,t}^{n-k}\right\Vert \\
 & \leqslant\sum_{k=0}^{\alpha-1}\left\Vert \mathbb{X}_{s,u}^{n-k}\right\Vert \cdot\left\Vert \mathbb{X}_{u,t}^{k}\right\Vert +\sum_{k=n-\alpha+1}^{n}\left\Vert \mathbb{X}_{s,u}^{n-k}\right\Vert \cdot\left\Vert \mathbb{X}_{u,t}^{k}\right\Vert \\
 & \ \ \ +\sum_{k=\alpha}^{n-\alpha}\left\Vert \mathbb{X}_{s,u}^{n-k}\right\Vert \cdot\left\Vert \mathbb{X}_{u,t}^{k}\right\Vert .
\end{align*}
Define $(s,t)\mapsto\widetilde{l}_{s,t}^{\alpha}\triangleq\sup_{k\geqslant\alpha}\left\Vert (k/p)!\mathbb{X}_{s,t}^{k}\right\Vert ^{p/k}.$
It follows that 
\begin{align*}
\left\Vert \mathbb{X}_{s,t}^{n}\right\Vert  & \leqslant\sum_{k=0}^{\alpha-1}\frac{\left(\widetilde{l}_{s,u}^{\alpha}\right)^{\frac{n-k}{p}}}{\left(\frac{n-k}{p}\right)!}\cdot\left\Vert \mathbb{X}_{u,t}^{k}\right\Vert +\sum_{k=n-\alpha+1}^{n}\frac{\left(\widetilde{l}_{u,t}^{\alpha}\right)^{\frac{k}{p}}}{\left(\frac{k}{p}\right)!}\cdot\|\mathbb{X}_{s,u}^{n-k}\|\\
 & \ \ \ +\sum_{k=\alpha}^{n-\alpha}\frac{\left(\widetilde{l}_{s,u}^{\alpha}\right)^{\frac{n-k}{p}}}{\left(\frac{n-k}{p}\right)!}\cdot\frac{\left(\widetilde{l}_{u,t}^{\alpha}\right)^{\frac{k}{p}}}{\left(\frac{k}{p}\right)!}\\
 & \leqslant\sum_{k=0}^{\alpha-1}\frac{\left(\widetilde{l}_{s,u}^{\alpha}\right)^{\frac{n-k}{p}}}{\left(\frac{n-k}{p}\right)!}\cdot\left\Vert \mathbb{X}_{u,t}^{k}\right\Vert +\sum_{k=n-\alpha+1}^{n}\frac{\left(\widetilde{l}_{u,t}^{\alpha}\right)^{\frac{k}{p}}}{\left(\frac{k}{p}\right)!}\cdot\|\mathbb{X}_{s,u}^{n-k}\|\\
 & \ \ \ +p\frac{\left(\widetilde{l}_{s,u}^{\alpha}+\widetilde{l}_{u,t}^{\alpha}\right)^{\frac{n}{p}}}{\left(\frac{n}{p}\right)!},
\end{align*}
where in the final inequality we have used the neo-classical inequality
(c.f. \cite{HH10}), which states that

\[
\sum_{i=0}^{N}\frac{a^{\frac{i}{p}}b^{\frac{N-i}{p}}}{\left(\frac{i}{p}\right)!\left(\frac{N-i}{p}\right)!}\leqslant p\frac{(a+b)^{\frac{N}{p}}}{\left(\frac{N}{p}\right)!},\ \ \ \forall a,b\geqslant0,p\geqslant1,N\in\mathbb{N}.
\]
Multiplying through by $(\frac{n}{p})!$ and using Lemma \ref{lem: factorial lemma}
which states that $(\frac{n}{p})!\leq Cn^{\frac{\alpha}{p}}(\frac{n-\alpha}{p})!$,
\begin{align*}
 & \left(\frac{n}{p}\right)!\left\Vert \mathbb{X}_{s,t}^{n}\right\Vert \\
 & \leqslant\frac{\left(\frac{n}{p}\right)!}{\left(\frac{n-\alpha}{p}\right)!}\left(\sum_{k=0}^{\alpha-1}\left(\widetilde{l}_{s,u}^{\alpha}\right)^{\frac{n-k}{p}}\cdot\left\Vert \mathbb{X}_{u,t}^{k}\right\Vert +\sum_{k=n-\alpha+1}^{n}\left(\widetilde{l}_{u,t}^{\alpha}\right)^{\frac{k}{p}}\cdot\left\Vert \mathbb{X}_{s,u}\right\Vert ^{n-k}\right)\\
 & \ \ \ +p\left(\widetilde{l}_{s,u}^{\alpha}+\widetilde{l}_{u,t}^{\alpha}\right)^{\frac{n}{p}}\\
 & \leqslant Cn^{\frac{\alpha}{p}}\left(\sum_{k=0}^{\alpha-1}\left(\widetilde{l}_{s,u}^{\alpha}\right)^{\frac{n-k}{p}}\cdot\left\Vert \mathbb{X}_{u,t}^{k}\right\Vert +\sum_{k=n-\alpha+1}^{n}\left(\widetilde{l}_{u,t}^{\alpha}\right)^{\frac{k}{p}}\cdot\left\Vert \mathbb{X}_{s,u}\right\Vert ^{n-k}\right)\\
 & \ \ \ +p\left(\widetilde{l}_{s,u}^{\alpha}+\widetilde{l}_{u,t}^{\alpha}\right)^{\frac{n}{p}}.
\end{align*}
Therefore, 
\begin{equation}
\widetilde{l}_{s,t}=\limsup_{n\rightarrow\infty}\left(\left(\frac{n}{p}\right)!\left\Vert \mathbb{X}_{s,t}^{n}\right\Vert \right)^{\frac{p}{n}}\leqslant\widetilde{l}_{s,u}^{\alpha}+\widetilde{l}_{u,t}^{\alpha},\label{eq:subadditive pre-limit}
\end{equation}
where we have used the simple fact that 
\[
\lim_{n\rightarrow\infty}\left(\left(\lambda a^{\frac{n}{p}}+\mu b^{\frac{n}{p}}\right)n^{\nu}+(a+b)^{\frac{n}{p}}\right)^{\frac{p}{n}}=a+b
\]
for any $\lambda,\mu,\nu,a,b,p>0$ (note that, as discussed at the
beginning of the proof, $\widetilde{l}_{s,u}^{\alpha},\widetilde{l}_{u,t}^{\alpha}>0$).

Now the result follows from taking $\alpha\rightarrow\infty$ in (\ref{eq:subadditive pre-limit}). 
\end{proof}
\begin{rem}
Typically if $\mathbf{X}$ has finite $p$-variation, then from Lyons'
extension theorem we know that $\widetilde{l}_{s,t}$ is finite. 
\end{rem}
\begin{rem}
This is a side remark related to whether the reverse inequality in
Lemma \ref{prop: sub-additivity} holds. Note that Lemma \ref{prop: sub-additivity}
holds for all geometric rough paths $\mathbf{X}$ regardless of whether
$\mathbf{X}$ has a tree-like piece (a tree-like piece is a loop in
an $\mathbb{R}$-tree, see \cite{BGLY16} or \cite{Geng17} for the
precise definition). On the other hand, a \textit{super}additive estimate
of the form 
\begin{equation}
\widetilde{l}_{s,t}\geqslant\widetilde{l}_{s,u}+\widetilde{l}_{u,t},\label{eq:superadditive}
\end{equation}
if true at all, can at best only hold for tree-reduced paths, as inserting
tree-like pieces could make the right hand side of (\ref{eq:superadditive})
arbitrarily big while leaving the left hand side unchanged. 
\end{rem}
\begin{thm}
\label{thm: the limsup is a constant}Let the tensor products over
$\mathbb{R}^{d}$ be equipped with given admissible norms, under which
each element of the standard basis $\{\mathrm{e}_{1},\cdots,\mathrm{e}_{d}\}$
of $\mathbb{R}^{d}$ has norm one. Then for each $s<t$, $\widetilde{L}_{s,t}$
is almost surely a deterministic constant which is bounded above by
$d^{2}(t-s)$. 
\end{thm}
\begin{proof}
For $m\geqslant1$, consider the dyadic partition 
\[
t_{i}^{m}\triangleq s+\frac{i}{2^{m}}(t-s),\ \ \ i=0,\cdots,2^{m}.
\]
According to the subadditivity estimate, Proposition \ref{prop: sub-additivity},
we know that pathwisely 
\[
\widetilde{L}_{s,t}\leqslant\sum_{i=1}^{2^{m}}\widetilde{L}_{t_{i-1}^{m},t_{i}^{m}}=2^{-m}\sum_{i=1}^{2^{m}}2^{m}\widetilde{L}_{t_{i-1}^{m},t_{i}^{m}}.
\]
On the one hand, by the Brownian scaling, for each $i$, $2^{m}\widetilde{L}_{t_{i-1}^{m},t_{i}^{m}}$
has the same distribution as $\widetilde{L}_{s,t}.$ In particular,
by Proposition \ref{prop: upper bound}, it is bounded above by $d^{2}(t-s)$
almost surely. On the other hand, the family $\{2^{m}\widetilde{L}_{t_{i-1}^{m},t_{i}^{m}}:\ 1\leqslant i\leqslant2^{m}\}$
are independent. According to the weak law of large numbers, we conclude
that 
\[
2^{-m}\sum_{i=1}^{2^{m}}2^{m}\widetilde{L}_{t_{i-1}^{m},t_{i}^{m}}\rightarrow\mathbb{\mathbb{E}}\left[\widetilde{L}_{s,t}\right]
\]
in probability. By taking an almost surely convergent subsequence,
we obtain that 
\[
\widetilde{L}_{s,t}\leqslant\mathbb{E}\left[\widetilde{L}_{s,t}\right]
\]
almost surely. This certainly implies that $\widetilde{L}_{s,t}=\mathbb{E}\left[\widetilde{L}_{s,t}\right]$
almost surely. 
\end{proof}
\begin{rem}
Although the fact of $\widetilde{L}_{s,t}$ being a deterministic
constant is a result of independent increments for Brownian motion,
it is not clear that any simple type of $0$-$1$ law argument could
apply. 
\end{rem}
\begin{cor}
\label{cor: L =00003D00003D k t non-uniformly}Under the assumption
of Theorem \ref{thm: the limsup is a constant}, there exists a constant
$\kappa_{d}$ depending on $d$, such that for each pair of $s<t,$
with probability one we have $\widetilde{L}_{s,t}=\kappa_{d}(t-s).$ 
\end{cor}
\begin{proof}
The result follows immediately from Theorem \ref{thm: the limsup is a constant}
and Brownian scaling. 
\end{proof}
\begin{rem}
We should emphasize that the constant $\kappa_{d}$ depends on the
choice of given admissible norms on the tensor products. 
\end{rem}
We can further show that the $\mathbb{P}$-null set arising from Corollary
\ref{cor: L =00003D00003D k t non-uniformly} associated with each
pair of $s<t$ can be chosen to be universal. This point will be very
useful for applications to the level of the Brownian rough path (c.f.
Section 6 below).
\begin{prop}
\label{prop: universality of null set}With probability one, we have
\[
\widetilde{L}_{s,t}=\kappa_{d}(t-s)\ \ \ \mathrm{for}\ \mathrm{all}\ s<t.
\]
\end{prop}
\begin{proof}
According to Proposition \ref{prop: upper bound} and Corollary \ref{cor: L =00003D00003D k t non-uniformly},
there exists a $\mathbb{P}$-null set $\mathcal{N}$, such that for
all $\omega\notin\mathcal{N},$ we have 
\[
\max\left\{ L_{r_{1},r_{2}}'(\omega),\ L''_{r_{1},r_{2}}(\omega)\right\} \leqslant d^{2}(r_{2}-r_{1})
\]
and 
\[
\widetilde{L}_{r_{1},r_{2}}(\omega)=\kappa_{d}(r_{2}-r_{1})
\]
for all $r_{1},r_{2}\in\mathbb{Q}$ with $r_{1}<r_{2},$ where 
\begin{align*}
L'_{r_{1},r_{2}} & \triangleq\limsup_{n\rightarrow\infty}\left(\left(\frac{n}{2}\right)!\sup_{r_{1}\leqslant u\leqslant r_{2}}\left\Vert \mathbb{B}_{r_{1},u}^{n}\right\Vert \right)^{\frac{2}{n}},\\
L''_{r_{1},r_{2}} & \triangleq\limsup_{n\rightarrow\infty}\left(\left(\frac{n}{2}\right)!\sup_{r_{1}\leqslant u\leqslant r_{2}}\left\Vert \mathbb{B}_{u,r_{2}}^{n}\right\Vert \right)^{\frac{2}{n}}.
\end{align*}

Now fix $\omega\notin\mathcal{N}$ and let $s<r$ with $r\in\mathbb{Q}.$
For arbitrary $r_{1},r_{2}\in\mathbb{Q}$ with $r_{1}<s<r_{2},$ we
know that 
\begin{align*}
\kappa_{d}(r-r_{1}) & =\widetilde{L}_{r_{1},r}(\omega)\\
 & \leqslant\widetilde{L}_{r_{1},s}(\omega)+\widetilde{L}_{s,r}(\omega)\;\text{(by subadditivity estimate, Prop. \ref{prop: sub-additivity})}\\
 & \leqslant L'_{r_{1},r_{2}}(\omega)+\widetilde{L}_{s,r}(\omega)\\
 & \leqslant d^{2}(r_{2}-r_{1})+\widetilde{L}_{s,r}(\omega)\;\text{(by Prop. \ref{sec:upper bound})}
\end{align*}
By letting $r_{1}\uparrow s$ and $r_{2}\downarrow s$ along rational
times, we obtain that 
\[
\widetilde{L}_{s,r}(\omega)\geqslant\kappa_{d}(r-s).
\]
Similarly, from 
\begin{align*}
\widetilde{L}_{s,r}(\omega) & \leqslant\widetilde{L}_{s,r_{2}}(\omega)+\widetilde{L}_{r_{2},r}(\omega)\;\text{(by subadditivity estimate, Prop. \ref{prop: sub-additivity})}\\
 & \leqslant L''_{r_{1},r_{2}}(\omega)+\widetilde{L}_{r_{2},r}(\omega)\\
 & \leqslant d^{2}(r_{2}-r_{1})+\kappa_{d}(r-r_{2}),\;\text{(by Prop. \ref{sec:upper bound})}
\end{align*}
we conclude that 
\[
\widetilde{L}_{s,r}(\omega)\leqslant\kappa_{d}(r-s).
\]
Therefore, 
\[
\widetilde{L}_{s,r}(\omega)=\kappa_{d}(r-s).
\]
By repeating the same argument to the parameter $r$, we conclude
that for all $s<t,$ 
\[
\widetilde{L}_{s,t}(\omega)=\kappa_{d}(t-s).
\]
\end{proof}

\section{The second part of the main result: the lower estimate}

For given admissible tensor norms, from the last section we know that
with probability one, 
\[
\widetilde{L}_{s,t}=\kappa_{d}(t-s),\ \ 
\]
where $\kappa_{d}$ is a deterministic constant depending only on
the dimension $d$ of Brownian motion, which is bounded above by $d^{2}.$
It is not even clear that $\kappa_{d}$ should be strictly positive.
In this section, we are going to establish a lower estimate of $\kappa_{d}$
under the projective tensor norm by applying the technique of hyperbolic
development which was introduced by Hambly and Lyons' paper \cite{HL10}.
In the next section, we shall see that the positivity of $\kappa_{d}$
reflects certain non-degeneracy properties of the Brownian rough path.

\subsection{The hyperbolic development of a regular path}

Before studying the Brownian signature, let us first summarize the
fundamental idea of hyperbolic development in the deterministic context
for regular paths. We present proofs of a few results which seems
not appearing in the literature. For an expository review on hyperbolic
geometry, we refer the reader to the wonderful survey \cite{CFKP97}.

Let $\mathbb{H}^{d}$ $(d\geqslant2)$ be the complete, connected
and simply-connected $d$-dimensional Riemannian manifold with constant
sectional curvature $-1.$ For computational convenience, we choose
the hyperboloid model. In particular, $\mathbb{H}^{d}$ is defined
to be the submanifold $\{x\in\mathbb{R}^{d+1}:\ x*x=-1,\ x^{d+1}>0\},$
where $*$ is the Minkowski metric on $\mathbb{R}^{d+1}$ given by
\[
x*y\triangleq\sum_{i=1}^{d}x^{i}y^{i}-x^{d+1}y^{d+1}.
\]
The Minkowski metric induces a Riemannian metric on $\mathbb{H}^{d}$
which gives it the desired hyperbolic structure. For $x,y\in\mathbb{H}^{d},$
one can show that 
\begin{equation}
\cosh\rho(x,y)=-x*y,\label{eq: relationship between Lorenz metric and hyperbolic metric}
\end{equation}
where $\rho(x,y)$ is the hyperbolic distance between $x$ and $y.$

It is known that the isometry group $\mathrm{SO}(d,1)$ of $\mathbb{H}^{d}$
is the space of $(d+1)\times(d+1)$-invertible matrices $\Gamma$
such that $\Gamma^{-1}=J\Gamma^{*}J$ and $\Gamma_{d+1}^{d+1}>0,$
where $J\triangleq\mathrm{diag}(1,\cdots,1,-1).$ The Lie algebra
$\mathrm{so}(d,1)$ of $\mathrm{SO}(d,1)$ is the space of $(d+1)\times(d+1)$-matrices
$A$ of the form 
\[
A=\left(\begin{array}{cc}
A_{0} & b\\
b^{*} & 0
\end{array}\right)
\]
where $A_{0}$ is an antisymmetric $d\times d$-matrix and $b\in\mathbb{R}^{d}.$

Define a linear map $F:\ \mathbb{R}^{d}\rightarrow\mathrm{so}(d,1)$
by 
\[
F(x)\triangleq\left(\begin{array}{cccc}
0 & \cdots & 0 & x^{1}\\
\vdots & \ddots & \vdots & \vdots\\
0 & \cdots & 0 & x^{d}\\
x^{1} & \cdots & x^{d} & 0
\end{array}\right),\ \ \ x=(x^{1},\cdots,x^{d})\in\mathbb{R}^{d}.
\]
Given a continuous path $\gamma:\ [0,1]\rightarrow\mathbb{R}^{d}$
with bounded variation, consider the linear ordinary differential
equation 
\[
\begin{cases}
d\Gamma_{t}=\Gamma_{t}F(\mathrm{d}\gamma_{t}), & t\in[0,1],\\
\Gamma_{0}=\mathrm{I}_{d+1}.
\end{cases}
\]
The solution $\Gamma_{t}$ defines a continuous path with bounded
variation in the isometry group $\mathrm{SO}(d,1).$ Explicitly, by
Picard's iteration, we see that 
\begin{equation}
\Gamma_{t}=\sum_{n=0}^{\infty}\int_{0<t_{1}<\cdots<t_{n}<t}F(\mathrm{d}\gamma_{t_{1}})\cdots F(\mathrm{d}\gamma_{t_{n}})=\sum_{n=0}^{\infty}F^{\otimes n}\left(g_{n}(t)\right),\label{eq: explicit formula for Gamma}
\end{equation}
where we recall that $g_{n}(t)\triangleq\int_{0<t_{1}<\ldots<t_{n}<t}\mathrm{d}\gamma_{t_{1}}\otimes\ldots\otimes\mathrm{d}\gamma_{t_{n}}.$Define
$X_{t}\triangleq\Gamma_{t}o,$ where $o=(0,\cdots,0,1)^{*}\in\mathbb{H}^{d}.$
\begin{defn}
$\Gamma_{t}$ is called the \textit{Cartan development} of $\gamma_{t}$
onto $\mathrm{SO}(d,1).$ $X_{t}$ is called the \textit{hyperbolic
development} of $\gamma_{t}$ onto $\mathbb{H}^{d}.$ 
\end{defn}
The reason of expecting a lower estimate of $\kappa_{d}$ in our Brownian
setting from the hyperbolic development is quite related to the philosophy
in the setting of bounded variation paths. To be precise, define 
\begin{equation}
\widetilde{l}\triangleq\sup_{n\geqslant1}\left(n!\|g_{n}\|_{\mathrm{proj}}\right)^{\frac{1}{n}}\leqslant\|\gamma\|_{1\mathrm{-var}},\label{eq: definition of sup in the BV case}
\end{equation}
where $g_{n}\triangleq\int_{0<t_{1}<\cdots<t_{n}<1}\mathrm{d}\gamma_{t_{1}}\otimes\cdots\otimes\mathrm{d}\gamma_{t_{n}}$
is the $n$-th component of the signature of $\gamma,$ and $\|\cdot\|_{\mathrm{proj}}$
is the projective tensor norm induced by the Euclidean norm on $\mathbb{R}^{d}.$ 

Now suppose that $\gamma$ is tree-reduced. There are essentially
two cases in which the length conjecture $\widetilde{l}=\|\gamma\|_{1-\mathrm{var}}$
is known to be true: piecewise linear paths or $C^{1}$-paths in constant
speed parametrization (\cite{HL10,LX15}).

The fundamental reason that the hyperbolic development yields the
lower bound $\widetilde{l}\geqslant\|\gamma\|_{1-\mathrm{var}}$ is
hidden in the following two key facts.

\textbf{Fact 1.} The hyperbolic development is length preserving.
Moreover, if $\gamma_{t}$ is piecewise linear, then its hyperbolic
development $X_{t}$ is piecewise geodesic with the same intersection
angles as those of $\gamma_{t}$. 
\begin{proof}
We first show that the Cartan development is length preserving.

If $\gamma_{t}$ is smooth, then the equation for $\Gamma_{t}$ becomes
\[
\stackrel{\cdot}{\Gamma}_{t}=\Gamma_{t}F(\stackrel{\cdot}{\gamma}_{t}),
\]
and thus 
\[
\stackrel{\cdot}{X_{t}}=\stackrel{\cdot}{\Gamma_{t}}o=\Gamma_{t}\left(\begin{array}{c}
\stackrel{\cdot}{\gamma_{t}}\\
0
\end{array}\right).
\]
Since $\Gamma_{t}$ is an isometry of $\mathbb{H}^{d},$ by identifying
$T_{o}\mathbb{H}^{d}\cong\mathbb{R}^{d},$ we conclude that 
\[
\|\stackrel{\cdot}{X_{t}}\|_{*}=\left\Vert \left(\begin{array}{c}
\stackrel{\cdot}{\gamma}_{t}\\
0
\end{array}\right)\right\Vert _{*}=\|\stackrel{\cdot}{\gamma_{t}}\|_{\mathrm{Euclidean}},
\]
where we define \textit{
\[
\left\Vert \left(\begin{array}{c}
x_{1}\\
\vdots\\
x_{d+1}
\end{array}\right)\right\Vert _{*}\triangleq\sqrt{\left(\begin{array}{c}
x_{1}\\
\vdots\\
x_{d+1}
\end{array}\right)*\left(\begin{array}{c}
x_{1}\\
\vdots\\
x_{d+1}
\end{array}\right)}=\sqrt{\sum_{i=1}^{d}x_{i}^{2}-x_{d+1}^{2}}.
\]
}It follows that the hyperbolic length of $X_{t}$ is the same as
the Euclidean length of $\gamma_{t}$. The general bounded variation
case can be proved by smooth approximation.

Next we show that the Cartan development of a piecewise linear path
is a piecewise geodesic with the same intersection angles.

If $\gamma_{t}=tv$ is a linear path, it can be shown using (\ref{eq: explicit formula for Gamma})
that 
\begin{equation}
X_{1}^{d+1}=(\Gamma_{1}o)^{d+1}=\sum_{n=0}^{\infty}\frac{\|v\|_{\mathrm{Euclidean}}^{2n}}{(2n)!}=\cosh\left\Vert v\right\Vert _{\mathrm{Euclidean}}.\label{eq:equality}
\end{equation}
(As this equality is given for motivation only, we will not give a
proof.) From the identity (\ref{eq: relationship between Lorenz metric and hyperbolic metric}),
we know that 
\[
\cosh\rho(X_{1},o)=-X_{1}*o=X_{1}^{d+1}=\cosh\|v\|_{\mathrm{Euclidean}},
\]
which implies that 
\[
\rho(X_{1},o)=\|v\|_{\mathrm{Euclidean}}=\|\gamma\|_{1-\mathrm{var}}.
\]
Therefore, $X$ is a geodesic in $\mathbb{H}_{d}.$

Now suppose that $\gamma_{t}$ is piecewise linear over a partition
$\mathcal{P}:\ 0=t_{0}<t_{1}<\cdots<t_{n+1}=1,$ where $\stackrel{\cdot}{\gamma}_{t}=v_{k}\in\mathbb{R}^{d}$
for $t\in[t_{k-1},t_{k}].$ Apparently, the Cartan development $X_{t}$
of $\gamma_{t}$ is a piecewise geodesic. Given $1\leqslant k\leqslant n,$
we have 
\[
\stackrel{\cdot}{X}_{t_{k}-}=\Gamma_{t_{k-1}}\Gamma_{t_{k-1}}^{-1}\Gamma_{t_{k}}\left(\begin{array}{c}
v_{k}\\
0
\end{array}\right)=\Gamma_{t_{k}}\left(\begin{array}{c}
v_{k}\\
0
\end{array}\right),
\]
and 
\[
\stackrel{\cdot}{X}_{t_{k}+}=\Gamma_{t_{k}}\left(\begin{array}{c}
v_{k+1}\\
0
\end{array}\right).
\]
Therefore, 
\[
\langle v_{k},v_{k+1}\rangle_{\mathrm{Euclidean}}=\left\langle \left(\begin{array}{c}
v_{k}\\
0
\end{array}\right),\left(\begin{array}{c}
v_{k+1}\\
0
\end{array}\right)\right\rangle _{*}=\left\langle \stackrel{\cdot}{X}_{t_{k}-},\stackrel{\cdot}{X}_{t_{k}+}\right\rangle _{*},
\]
where the second equality uses that $\Gamma_{t_{k}}$ is an isometry
with respect to $\left\langle \cdot,\cdot\right\rangle _{*}$ and
we define \textit{
\[
\left\langle \left(\begin{array}{c}
x_{1}\\
\vdots\\
x_{d+1}
\end{array}\right),\left(\begin{array}{c}
y_{1}\\
\vdots\\
y_{d+1}
\end{array}\right)\right\rangle _{*}\triangleq\left(\begin{array}{c}
x_{1}\\
\vdots\\
x_{d+1}
\end{array}\right)*\left(\begin{array}{c}
y_{1}\\
\vdots\\
y_{d+1}
\end{array}\right)=\sum_{i=1}^{d}x_{i}y_{i}-x_{d+1}y_{d+1}.
\]
}This implies that the Cartan development preserves intersection angles. 
\end{proof}
\textbf{Fact 2.} In a hyperbolic triangle with edges $a,b,c>0,$ we
have $a\geqslant b+c-\log\frac{2}{1-\cos\theta_{A}}$, where $\theta_{A}$
is the angle opposite $a.$
\begin{proof}
The only point which requires some attention is the following fact:
for $\lambda>0,$ if we consider triangles with the same angle $\theta_{A}$
(its opposite edge being denoted by $a(\lambda)$), and $\lambda b,\lambda c$
being the other two edges, then 
\[
f(\lambda)\triangleq\lambda b+\lambda c-a(\lambda)
\]
is monotonely increasing in $\lambda$. Based on this fact, one finds
the upper bound of $b+c-a$ to be $\lim_{\lambda\rightarrow\infty}(\lambda b+\lambda c-a(\lambda))$,
which can be computed by using the hyperbolic cosine law (c.f. Proof
of Lemma 3.4 in \cite{HL10})

To this end, it suffices to show that $f'(\lambda)=b+c-a'(\lambda)\geqslant0.$
By the first hyperbolic cosine law, we have 
\begin{equation}
\cosh a(\lambda)=\cosh\lambda b\cosh\lambda c-\sinh\lambda b\sinh\lambda c\cos\theta_{A}.\label{eq:hyperbolic cosine law}
\end{equation}
Differentiating with respect to $\lambda,$ we obtain that 
\begin{align*}
a'(\lambda)\sinh a(\lambda) & =b\left(\sinh\lambda b\cosh\lambda c-r\cosh\lambda b\sinh\lambda c\right)\\
 & \ \ \ +c\left(\cosh\lambda b\sinh\lambda c-r\sinh\lambda b\cosh\lambda c\right)
\end{align*}
where $r\triangleq\cos\theta_{A}.$ For simplicity we write $\sinh=\mathrm{sh},$
$\cosh=\mathrm{ch.}$ Now it suffices to show that 
\[
b(\mathrm{sh}\lambda b\cdot\mathrm{ch}\lambda c-r\mathrm{ch}\lambda b\cdot\mathrm{sh}\lambda c)+c(\mathrm{ch}\lambda b\cdot\mathrm{sh}\lambda c-r\mathrm{sh}\lambda b\cdot\mathrm{ch}\lambda c)\leqslant(b+c)\mathrm{sh}a(\lambda).
\]

We use $X$, $Y$ to denote the left and right hand sides respectively.
From direct computation, we see that 
\begin{align*}
X^{2} & =(b-cr)^{2}\mathrm{sh}^{2}\lambda b\cdot\mathrm{ch}^{2}\lambda c+(c-br)^{2}\mathrm{ch}^{2}\lambda b\cdot\mathrm{sh}^{2}\lambda c\\
 & \ \ \ +(2bc+2bcr^{2}-2b^{2}r-2c^{2}r)\mathrm{sh}\lambda b\cdot\mathrm{ch}\lambda b\cdot\mathrm{sh}\lambda c\cdot\mathrm{ch}\lambda c,
\end{align*}
and by the hyperbolic cosine law (\ref{eq:hyperbolic cosine law}),
\begin{align*}
Y^{2} & =(b+c)^{2}((1+r^{2})\mathrm{sh}^{2}\lambda b\cdot\mathrm{sh}^{2}\lambda c+\mathrm{sh}^{2}\lambda b+\mathrm{sh}^{2}\lambda c\\
 & \ \ \ -2r\mathrm{sh}\lambda b\cdot\mathrm{ch}\lambda b\cdot\mathrm{sh}\lambda c\cdot\mathrm{ch}\lambda c).
\end{align*}
By using $\cosh^{2}x-\sinh^{2}x=1,$ we obtain that 
\begin{align*}
\frac{Y^{2}-X^{2}}{1+r} & =2bc(1+r)\mathrm{sh}^{2}\lambda b\cdot\mathrm{sh}^{2}\lambda c-2bc(1+r)\mathrm{sh}\lambda b\cdot\mathrm{ch}\lambda b\cdot\mathrm{sh}\lambda c\cdot\mathrm{ch}\lambda c\\
 & \ \ \ +(c^{2}(1-r)+2bc)\mathrm{sh}^{2}\lambda b+(b^{2}(1-r)+2bc)\mathrm{sh}^{2}\lambda c.
\end{align*}

Define $g(r)$ to be the function in $r$ given by the right hand
side of the above equality. Then 
\[
g(1)=2bc(\mathrm{sh}\lambda b\cdot\mathrm{ch}\lambda c-\mathrm{ch}\lambda b\cdot\mathrm{sh}\lambda c)^{2}\geqslant0.
\]
Moreover, 
\begin{align*}
g'(r) & =-2bc\mathrm{sh}\lambda b\cdot\mathrm{sh}\lambda c\cdot\mathrm{ch}\lambda(b-c)\\
 & \ \ \ -c^{2}\mathrm{sh}^{2}\lambda b-b^{2}\mathrm{sh}^{2}\lambda c\\
 & \leqslant0,
\end{align*}
where the inequality in the final line follows by using $\mathrm{ch}\lambda(b-c)\geqslant1$
and completing the square. Therefore, $g(r)\geqslant0$ for $r\in[-1,1],$
which implies that $Y^{2}\geqslant X^{2}.$ Since $Y\geqslant0,$
we conclude that $Y\geqslant X.$ 
\end{proof}
Let $\gamma:\ [0,1]\rightarrow\mathbb{R}^{d}$ be a tree-reduced bounded
variation path. From (\ref{eq: relationship between Lorenz metric and hyperbolic metric})
and the explicit formula for the Cartan development, it can be shown
that (see also (\ref{eq: formula for hyperbolic height}) below),
for each $\lambda>0,$ 
\begin{align*}
\cosh\rho(X_{1}^{\lambda},o) & =\sum_{n=0}^{\infty}\lambda^{2n}\int_{0<t_{1}<\cdots<t_{2n}<1}\langle\mathrm{d}\gamma_{t_{1}},\mathrm{d}\gamma_{t_{2}}\rangle\cdots\langle\mathrm{d}\gamma_{t_{2n-1}},\mathrm{d}\gamma_{t_{2n}}\rangle\\
 & \leqslant\sum_{n=0}^{\infty}\lambda^{2n}\left\Vert \int_{0<t_{1}<\cdots<t_{2n}<1}\mathrm{d}\gamma_{t_{1}}\otimes\cdots\otimes\mathrm{d}\gamma_{t_{2n}}\right\Vert _{\text{proj}}\quad(\text{see}\;(\ref{eq:inner product dominated by projective norm})\;\text{below})\\
 & \leqslant\cosh\lambda\widetilde{l}.
\end{align*}
where $\widetilde{l}$ is defined by (\ref{eq: definition of sup in the BV case})
as the supremum of the (normalized) iterated integrals, $\gamma_{t}^{\lambda}\triangleq\lambda\gamma_{t}$
($0\leqslant t\leqslant1$) is the path obtained by rescaling $\gamma$
by the factor $\lambda,$ and $X_{t}^{\lambda}$ is the hyperbolic
development of $\gamma_{t}^{\lambda}.$ In particular, we see that
$\lambda\widetilde{l}\geqslant\rho(X_{1}^{\lambda},o).$

The previous Fact 2 tells us that for all two-edge piecewise geodesic
paths $Y:\ [0,1]\rightarrow\mathbb{H}^{d}$ with fixed intersection
angle $0<\theta<\pi$, the distance between hyperbolic length of $Y$
and $\rho(Y_{1},o)$ is uniformly bounded by a constant depending
on $\theta.$ Now suppose that $\gamma:\ [0,1]\rightarrow\mathbb{R}^{d}$
is a two-edge piecewise linear path with intersection angle $0<\theta<\pi.$
Fact 1 and 2 together implies that 
\[
0\leqslant\lambda\|\gamma\|_{1-\mathrm{var}}-\rho\left(X_{1}^{\lambda},o\right)\leqslant K(\theta)\triangleq\log\frac{2}{1-\cos\theta},
\]
uniformly in $\lambda>0.$ In particular, 
\begin{equation}
\lim_{\lambda\rightarrow\infty}\frac{\rho\left(X_{1}^{\lambda},o\right)}{\lambda}=\|\gamma\|_{1-\mathrm{var}},\label{eq: limit of hyperbolic distance}
\end{equation}
from which we obtain the desired estimate $\widetilde{l}\geqslant\|\gamma\|_{1-\mathrm{var}}.$
It is important to note that the angle $\theta$ captures the tree-reduced
nature of $\gamma$ in this simple case. Indeed, if $\theta=0,$ $K(\theta)=+\infty.$

With some effort, the previous argument extends to tree-reduced piecewise
linear paths with minimal intersection angle given by $\theta>0.$
In this case, one can obtain an estimate of the form 
\[
0\leqslant\lambda\|\gamma\|_{1-\mathrm{var}}-\rho\left(X_{1}^{\lambda},o\right)\leqslant N\cdot\Lambda(\theta)
\]
uniformly in $\lambda>0,$ where $N$ is the number of edges of $\gamma$
and $\Lambda(\theta)$ is a constant depending only $\theta$ (which
explodes as $\theta\downarrow0$). We again obtain (\ref{eq: limit of hyperbolic distance})
and thus the desired estimate. Here $\theta>0$ captures the tree-reduced
nature of $\gamma.$ With some further delicate analysis, one can
establish a similar estimate for a path $\gamma:\ [0,1]\rightarrow\mathbb{R}^{d}$
which is continuously differentiable when parametrized at constant
speed. The estimate takes the form 
\[
0\leqslant\lambda\|\gamma\|_{1-\mathrm{var}}-\rho(X_{1}^{\lambda},o)\leqslant C_{1}\lambda\|\gamma\|_{1-\mathrm{var}}\delta_{\gamma}\left(\frac{C_{2}}{\lambda}\right)^{2}
\]
provided that $\lambda$ is large, where $C_{1},C_{2}$ are universal
constants and $\delta_{\gamma}(\cdot)$ is the modulus of continuity
for $\dot{\gamma}.$ In particular, we again obtain (\ref{eq: limit of hyperbolic distance})
and thus the desired estimate. Here the existence of modulus of continuity
for the derivative $\dot{\gamma}$ already implies that $\gamma$
is tree-reduced implicitly. In any case, the fundamental reason which
makes the technique of hyperbolic development work is hidden in the
nature of Fact 1 and 2.

If one is attempting to attack the length conjecture $\widetilde{l}=\|\gamma\|_{1-\mathrm{var}}$
for a general tree-reduced path with bounded variation by using the
idea of hyperbolic development, it seems that a crucial point is to
find a quantity $\omega_{\gamma},$ a certain kind of ``modulus of
continuity'', which on the one hand captures the tree-reduced nature
of $\gamma$ quantitatively, and on the other hand can be used to
control the growth of $\lambda\mapsto\lambda\|\gamma\|_{1-\mathrm{var}}-\rho(X_{1}^{\lambda},o)$
(difference between hyperbolic length and hyperbolic distance for
the rescaled path). Up to the current point, this fascinating and
challenging problem remains unsolved.

\subsection{The hyperbolic development of Brownian motion and a lower estimate
for $\kappa_{d}$}

In spite of the huge difficulty in obtaining lower estimates of the
hyperbolic distance function in the general deterministic setting,
it is surprising that a simple martingale argument will give us a
meaningful lower estimate for the hyperbolic development of Brownian
motion. In particular, we can obtain a lower estimate on the constant
$\kappa_{d}$.

From now on, we assume that $\mathbb{R}^{d}$ is equipped with the
$l^{p}$-norm for some given $1\leqslant p\leqslant2,$ and the tensor
products over $\mathbb{R}^{d}$ are equipped with the associated projective
tensor norms.

The following characterization of projective tensor norms is important
for us. 
\begin{lem}
\label{lem: dual characterization of projective tensor norm}For each
$\xi\in\left(\mathbb{R}^{d}\right)^{\otimes n},$ we have 
\[
\|\xi\|_{\mathrm{proj}}=\sup\left\{ \left|\Phi(\xi)\right|:\ \Phi\in L(\mathbb{R}^{d},\cdots,\mathbb{R}^{d};\mathbb{R}^{1}),\ \|\Phi\|\leqslant1\right\} ,
\]
where we identify $L(\mathbb{R}^{d},\cdots,\mathbb{R}^{d};\mathbb{R}^{1})$
with $\left((\mathbb{R}^{d})^{\otimes n}\right)^{*}$ through the
universal property, and 
\[
\|\Phi\|\triangleq\inf\{C\geqslant0:\ |\Phi(v_{1},\cdots,v_{n})|\leqslant C\|v_{1}\|\cdots\|v_{n}\|\ \ \forall v_{1},\cdots,v_{n}\in\mathbb{R}^{d}\}.
\]
\end{lem}
\begin{proof}
See \cite{Ryan16}, Identity (2.3). 
\end{proof}
Let $B_{t}=(B_{t}^{1},\cdots,B_{t}^{d})$ be a $d$-dimensional Brownian
motion. We define 
\[
\widetilde{L}_{t}\triangleq\limsup_{n\rightarrow\infty}\left(\left(\frac{n}{2}\right)!\|\mathbb{B}_{0,t}^{n}\|_{\mathrm{proj}}\right)^{\frac{2}{n}}.
\]

For each $\lambda>0,$ we consider the Cartan development 
\begin{equation}
\begin{cases}
d\Gamma_{t}^{\lambda}=\lambda\Gamma_{t}^{\lambda}F(\circ\mathrm{d}B_{t}), & t\geqslant0,\\
\Gamma_{0}^{\lambda}=\mathrm{I}_{d+1},
\end{cases}\label{eq:Rescaled Cartan Development}
\end{equation}
of $\lambda\cdot B_{t}$, where the differential equation is understood
in the Stratonovich sense. Let $X_{t}^{\lambda}\triangleq\Gamma_{t}^{\lambda}o$
be the hyperbolic development of $\lambda\cdot B_{t}.$ As in the
Cartan development driven by bounded variation paths, $\Gamma_{t}^{\lambda}$
also defines a path on the isometry group $\mathrm{SO}(d,1)$ and
hence $X_{t}^{\lambda}$ is a path on $\mathbb{H}^{d}$ starting at
$o$.

Picard's iteration again shows that 
\begin{equation}
\Gamma_{t}^{\lambda}=\sum_{n=0}^{\infty}\lambda^{n}\int_{0<t_{1}<\cdots<t_{n}<t}F(\circ\mathrm{d}B_{t_{1}})\cdots F(\circ\mathrm{d}B_{t_{n}}).\label{eq: Cartan development of BM}
\end{equation}
Define $h_{t}^{\lambda}\triangleq\left(X_{t}^{\lambda}\right)^{d+1}$
to be the hyperbolic height of $X_{t}^{\lambda}$ (the last coordinate
of $X_{t}^{\lambda}$). It follows from (\ref{eq: relationship between Lorenz metric and hyperbolic metric}),
(\ref{eq: Cartan development of BM}) and the definition of $F$ that
\begin{align}
h_{t}^{\lambda} & =\cosh\rho(X_{t}^{\lambda},o)\nonumber \\
 & =\sum_{n=0}^{\infty}\lambda^{n}\int_{0<t_{1}<\cdots<t_{n}<t}\left(F(\circ\mathrm{d}B_{t_{1}})\cdots F(\circ\mathrm{d}B_{t_{n}})o\right)^{d+1}\nonumber \\
 & =\sum_{n=0}^{\infty}\lambda^{2n}\int_{0<t_{1}<\cdots<t_{n}<t}\langle\circ\mathrm{d}B_{t_{1}},\circ\mathrm{d}B_{t_{2}}\rangle\cdots\langle\circ\mathrm{d}B_{t_{2n-1}},\circ\mathrm{d}B_{t_{2n}}\rangle,\label{eq: formula for hyperbolic height}
\end{align}

The following result shows that the quantity $\widetilde{L}_{t}$
can be controlled from below in terms of the asymptotics of $h_{t}^{\lambda}$
as $\lambda\rightarrow\infty.$
\begin{prop}
\label{prop: lower estimate in terms of Brownian hyperbolic height}With
probability one, we have 
\[
\limsup_{\lambda\rightarrow\infty}\frac{1}{\lambda^{2}}\log h_{t}^{\lambda}\leqslant\widetilde{L}_{t},\ \ \ \forall t\geqslant0.
\]
\end{prop}
\begin{proof}
For each $n\geqslant1,$ define a real-valued $2n$-linear map $\Phi_{n}$
over $\mathbb{R}^{d}$ by 
\[
\Phi_{n}(v_{1},\cdots,v_{2n})\triangleq\langle v_{1},v_{2}\rangle\cdots\langle v_{2n-1},v_{2n}\rangle,\ \ \ v_{1},\cdots,v_{2n}\in\mathbb{R}^{d}.
\]
Since we are taking the $l^{p}$-norm on $\mathbb{R}^{d}$ for $1\leqslant p\leqslant2,$
we see that 
\begin{align*}
\left|\Phi_{n}(v_{1},\cdots,v_{2n})\right| & \leqslant\|v_{1}\|_{l^{2}}\cdots\|v_{2n}\|_{l^{2}}\\
 & \leqslant\|v_{1}\|_{l^{p}}\cdots\|v_{2n}\|_{l^{p}}.
\end{align*}
In particular, $\|\Phi_{n}\|\leqslant1.$ Therefore, by Lemma \ref{lem: dual characterization of projective tensor norm},
we have 
\begin{equation}
\left|\int_{0<t_{1}<\cdots<t_{n}<t}\langle\circ\mathrm{d}B_{t_{1}},\circ\mathrm{d}B_{t_{2}}\rangle\cdots\langle\circ\mathrm{d}B_{t_{2n-1}},\circ\mathrm{d}B_{t_{2n}}\rangle\right|=|\Phi_{n}(\mathbb{B}_{0,t}^{2n})|\leqslant\|\mathbb{B}_{0,t}^{2n}\|_{\mathrm{proj}}.\label{eq:inner product dominated by projective norm}
\end{equation}

Now for each $\alpha\geqslant1,$ define 
\[
\widetilde{L}_{t}^{\alpha}\triangleq\sup_{n\geqslant\alpha}\left(\left(\frac{n}{2}\right)!\|\mathbb{B}_{0,t}^{n}\|_{\mathrm{proj}}\right)^{\frac{2}{n}}.
\]
It follows from (\ref{eq:Rescaled Cartan Development}) and (\ref{eq: Cartan development of BM})
that 
\begin{align*}
h_{t}^{\lambda} & \leqslant\sum_{n=0}^{\infty}\lambda^{2n}\|\mathbb{B}_{0,t}^{2n}\|_{\mathrm{proj}}\\
 & =\sum_{n=0}^{\alpha-1}\lambda^{2n}\|\mathbb{B}_{0,t}^{2n}\|_{\mathrm{proj}}+\sum_{n=\alpha}^{\infty}\lambda^{2n}\|\mathbb{B}_{0,t}^{2n}\|_{\mathrm{proj}}\\
 & \leqslant\sum_{n=0}^{\alpha-1}\lambda^{2n}\|\mathbb{B}_{0,t}^{2n}\|_{\mathrm{proj}}+\sum_{n=\alpha}^{\infty}\lambda^{2n}\cdot\frac{\left(\widetilde{L}_{t}^{2\alpha}\right)^{n}}{n!}\\
 & =\exp\left(\lambda^{2}\widetilde{L}_{t}^{2\alpha}\right)+\sum_{n=0}^{\alpha-1}\lambda^{2n}\left(\|\mathbb{B}_{0,t}^{2n}\|_{\mathrm{proj}}-\frac{\left(\widetilde{L}_{t}^{2\alpha}\right)^{n}}{n!}\right).
\end{align*}
Therefore, 
\begin{align*}
 & \limsup_{\lambda\rightarrow\infty}\frac{1}{\lambda^{2}}\log h_{t}^{\lambda}\\
 & \leqslant\limsup_{\lambda\rightarrow\infty}\frac{1}{\lambda^{2}}\log\left(\exp\left(\lambda^{2}\widetilde{L}_{t}^{2\alpha}\right)+\sum_{n=0}^{\alpha-1}\lambda^{2n}\left(\|\mathbb{B}_{0,t}^{2n}\|_{\mathrm{proj}}-\frac{\left(\widetilde{L}_{t}^{2\alpha}\right)^{n}}{n!}\right)\right)\\
 & =\widetilde{L}_{t}^{2\alpha}.
\end{align*}
Since $\alpha$ is arbitrary, we conclude that 
\[
\limsup_{\lambda\rightarrow\infty}\frac{1}{\lambda^{2}}\log h_{t}^{\lambda}\leqslant\inf_{\alpha\geqslant1}\widetilde{L}_{t}^{2\alpha}=\widetilde{L}_{t}.
\]
\end{proof}
The following result is the probabilistic counterpart of a lower estimate
on the hyperbolic height function $h_{t}^{\lambda}.$ 
\begin{lem}
\label{lem: lower estimate on Brownian hyperbolic height}For any
$0<\mu<d-1,$ we have 
\[
\mathbb{E}\left[\left(h_{t}^{\lambda}\right)^{-\mu}\right]\leqslant\exp\left(-\frac{\lambda^{2}\mu(d-1-\mu)}{2}t\right).
\]
\end{lem}
\begin{proof}
Throughout the rest of this paper, we will use $\cdot\mathrm{d}$
to denote the Itô differential. 

Note that 
\[
F(\mathrm{d}B_{t})\cdot F(\mathrm{d}B_{t})=\left(\begin{array}{cc}
0 & \mathrm{d}B_{t}\\
(\mathrm{d}B_{t})^{*} & 0
\end{array}\right)\cdot\left(\begin{array}{cc}
0 & \mathrm{d}B_{t}\\
(\mathrm{d}B_{t})^{*} & 0
\end{array}\right)=\left(\begin{array}{cc}
\mathrm{I}_{d} & 0\\
0 & d
\end{array}\right)\mathrm{d}t.
\]
Applying the Itô-Stratonovich conversion to the differential equation
for $\Gamma_{t}^{\lambda}$, we have 
\begin{align*}
\mathrm{d}\Gamma_{t}^{\lambda} & =\lambda\Gamma_{t}^{\lambda}\cdot F(\mathrm{d}B_{t})+\frac{\lambda}{2}d\Gamma_{t}^{\lambda}\cdot F(\mathrm{d}B_{t})\\
 & =\lambda\Gamma_{t}^{\lambda}\cdot F(\mathrm{d}B_{t})+\frac{\lambda^{2}}{2}\Gamma_{t}^{\lambda}\left(F(\mathrm{d}B_{t})\cdot F(\mathrm{d}B_{t})\right)\\
 & =\lambda\Gamma_{t}^{\lambda}\cdot F(\mathrm{d}B_{t})+\frac{\lambda^{2}}{2}\Gamma_{t}^{\lambda}\left(\begin{array}{cc}
\mathrm{I}_{d} & 0\\
0 & d
\end{array}\right)\mathrm{d}t.
\end{align*}
Therefore by restricting our attention to the $(d+1,d+1)$ coordinate
of the matrix, 
\begin{align}
\mathrm{d}h_{t}^{\lambda} & =\mathrm{d}\left(\Gamma_{t}^{\lambda}\right)_{d+1}^{d+1}\nonumber \\
 & =\lambda\sum_{i=1}^{d}\left(\Gamma_{t}^{\lambda}\right)_{i}^{d+1}\cdot\mathrm{d}B_{t}^{i}+\frac{\lambda^{2}d}{2}h_{t}^{\lambda}\mathrm{d}t.\label{eq:SDE for h}
\end{align}
Moreover, since $\Gamma_{t}^{\lambda}\in\mathrm{SO}(d,1),$ we know
that 
\begin{equation}
\sum_{i=1}^{d}\left(\left(\Gamma_{t}^{\lambda}\right)_{i}^{d+1}\right)^{2}-\left(h_{t}^{\lambda}\right)^{2}=-1,\label{eq:hyperbolic surface}
\end{equation}
and hence by (\ref{eq:SDE for h}) and (\ref{eq:hyperbolic surface}),
\[
\mathrm{d}h_{t}^{\lambda}\cdot\mathrm{d}h_{t}^{\lambda}=\lambda^{2}\sum_{i=1}^{d}\left(\left(\Gamma_{t}^{\lambda}\right)_{i}^{d+1}\right)^{2}\mathrm{d}t=\lambda^{2}\left(\left(h_{t}^{\lambda}\right)^{2}-1\right)\mathrm{d}t.
\]

Note that $h_{t}^{\lambda}\geqslant1$ and hence we may apply Itô's
formula to $\left(h_{t}^{\lambda}\right)^{-\mu}$ and then (\ref{eq:SDE for h})
to obtain, 
\begin{align*}
\mathrm{d}\left(h_{t}^{\lambda}\right)^{-\mu} & =-\mu\left(h_{t}^{\lambda}\right)^{-(\mu+1)}\cdot\mathrm{d}h_{t}^{\lambda}+\frac{\mu(\mu+1)}{2}\left(h_{t}^{\lambda}\right)^{-(\mu+2)}\left(\mathrm{d}h_{t}^{\lambda}\cdot\mathrm{d}h_{t}^{\lambda}\right)\\
 & =-\lambda\mu\left(h_{t}^{\lambda}\right)^{-(\mu+1)}\sum_{i=1}^{d}\left(\Gamma_{t}^{\lambda}\right)_{i}^{d+1}\cdot\mathrm{d}B_{t}^{i}\\
 & \ \ \ -\left(\frac{\lambda^{2}\mu(d-1-\mu)}{2}\left(h_{t}^{\lambda}\right)^{-\mu}+\frac{\lambda^{2}\mu(\mu+1)}{2}\left(h_{t}^{\lambda}\right)^{-(\mu+2)}\right)\mathrm{d}t.
\end{align*}
By taking expectation and differentiating with respect to $t,$ we
obtain that 
\begin{align*}
\frac{\mathrm{d}}{\mathrm{d}t}\mathbb{E}\left[\left(h_{t}^{\lambda}\right)^{-\mu}\right] & =-\frac{\lambda^{2}\mu(d-1-\mu)}{2}\mathbb{E}\left[\left(h_{t}^{\lambda}\right)^{-\mu}\right]\\
 & \ \ \ -\frac{\lambda^{2}\mu(\mu+1)}{2}\mathbb{E}\left[\left(h_{t}^{\lambda}\right)^{-(\mu+2)}\right]\\
 & \leqslant-\frac{\lambda^{2}\mu(d-1-\mu)}{2}\mathbb{E}\left[\left(h_{t}^{\lambda}\right)^{-\mu}\right],
\end{align*}
where in the final inequality we used that $h_{t}^{\lambda}\geqslant1$
(see for example (\ref{eq: formula for hyperbolic height})). By Gronwall's
inequality, we arrive at 
\[
\mathbb{E}\left[\left(h_{t}^{\lambda}\right)^{-\mu}\right]\leqslant\exp\left(-\frac{\lambda^{2}\mu(d-1-\mu)}{2}t\right).
\]
\end{proof}
Now we can state our main lower estimate on $\kappa_{d}.$
\begin{thm}
\label{thm: lower estimate of constant}Under the $l^{p}$-norm ($1\leqslant p\leqslant2$)
on $\mathbb{R}^{d}$ and the associated projective tensor norms on
the tensor products, we have 
\[
\kappa_{d}\geqslant\frac{d-1}{2}.
\]
\end{thm}
\begin{proof}
Fix $t>0,$ $\lambda>0$ and $0<\mu<d-1.$ According to Lemma \ref{lem: lower estimate on Brownian hyperbolic height}
(which we have just proved), for each $K>0$,

\begin{align*}
\mathbb{P}\left(h_{t}^{\lambda}\leqslant K\right) & =\mathbb{P}\left(\left(h_{t}^{\lambda}\right)^{-\mu}\geqslant K^{-\mu}\right)\\
 & \leqslant K^{\mu}\mathbb{E}\left[\left(h_{t}^{\lambda}\right)^{-\mu}\right]\\
 & \leqslant K^{\mu}\exp\left(-\frac{\lambda^{2}\mu(d-1-\mu)}{2}t\right).
\end{align*}

Let $s\in\mathbb{Q}$ such that $s<t$. Now for each $m\geqslant1$,
define $\lambda_{m}\triangleq m$ and 
\[
K_{m}\triangleq\exp\left(\frac{m^{2}(d-1-\mu)}{2}s\right).
\]
It follows that 
\[
\mathbb{P}\left(h_{t}^{\lambda_{m}}\leqslant K_{m}\right)\leqslant\exp\left(-\frac{m^{2}\mu(d-1-\mu)}{2}(t-s)\right).
\]
In particular, $\sum_{m=1}^{\infty}\mathbb{P}\left(h_{t}^{\lambda_{m}}\leqslant K_{m}\right)<\infty.$
By the Borel-Cantelli lemma, there exists a $\mathbb{P}$-null set
$\mathcal{N}(s,t,\mu),$ such that for any $\omega\notin\mathcal{N}(s,t,\mu),$
there exists $M(\omega)\geqslant1$ with 
\[
h_{t}^{\lambda_{m}}(\omega)>\exp\left(\frac{m^{2}(d-1-\mu)}{2}s\right),\ \ \ \forall m\geqslant M(\omega).
\]
Therefore, 
\[
\limsup_{m\rightarrow\infty}\frac{1}{m^{2}}\log h_{t}^{\lambda_{m}}(\omega)\geqslant\frac{d-1-\mu}{2}s.
\]
By enlarging the $\mathbb{P}$-null set through rationals $s\uparrow t$
and $\mu\downarrow0$, we conclude that 
\[
\limsup_{m\rightarrow\infty}\frac{1}{m^{2}}\log h_{t}^{\lambda_{m}}\geqslant\frac{d-1}{2}t
\]
for almost surely.

Finally, according to Proposition \ref{prop: lower estimate in terms of Brownian hyperbolic height}
which relates $\tilde{L}_{t}$ and $h_{t}$, we obtain that 
\[
\kappa_{d}=\frac{\widetilde{L}_{t}}{t}\geqslant\frac{d-1}{2}.
\]
\end{proof}

\section{Applications to the Brownian rough path}

We present a few interesting consequences of the lower estimate on
$\kappa_{d}$ given in Theorem \ref{thm: lower estimate of constant}.

Let us consider the $d$-dimensional Brownian motion $B_{t}$ on $[0,1].$
Recall that with probability one, $B_{t}$ has a canonical lifting
$\mathbf{B}_{t}$ as geometric $p$-rough path for $2<p<3.$ As a
process on $G^{2}(\mathbb{R}^{d}),$ the Brownian rough path $\mathbf{B}_{t}$
is canonically defined and it is independent of the choice of tensor
norms on $(\mathbb{R}^{d})^{\otimes2}.$
\begin{cor}
\label{cor: uniqueness of signature for Brownian rough path}If $d>1$,
then for almost every $\omega,$ the path $t\mapsto\mathbf{B}_{t}(\omega)$
is tree-reduced. 
\end{cor}
\begin{proof}
Let the tensor products be equipped with the projective tensor norms
associated with the $l^{2}$-norm on $\mathbb{R}^{d}.$ From Proposition
\ref{prop: universality of null set}, for every $\omega$ outside
some $\mathbb{P}$-null set $\mathcal{N}$, 
\begin{equation}
\widetilde{L}_{s,t}(\omega)=\kappa_{d}(t-s)\ \ \ \forall s<t.\label{eq:reparametrization}
\end{equation}
In addition, from Theorem \ref{thm: lower estimate of constant} we
know that the constant $\kappa_{d}$ is strictly positive and hence
$\tilde{L}_{s,t}>0$ for $s<t$. This implies that for every $\omega\notin\mathcal{N},$
the signature of $\mathbf{B}(\omega)$, $\mathbb{B}_{s,t}$, is non-trivial
for all $s<t$, which according to \cite{BGLY16} is the definition
of $\mathbf{B}$ being tree-reduced. 
\end{proof}
It was first proved by Le Jan and Qian \cite{LQ13} (see also \cite{BG15})
that the Stratonovitch signature of Brownian sample paths determine
the samples paths almost surely. Using Corollary \ref{cor: uniqueness of signature for Brownian rough path},
we obtain a stronger result below which explicitly reconstruct the
sample paths as well as its parametrization. 
\begin{cor}
\label{cor: strong uniqueness result for BRP}If $d>1$, then with
probability one, Brownian rough path together with its parametrization
can be recovered from its signature. 
\end{cor}
\begin{proof}
We define an equivalence relation $\sim$ on rough paths so that $\mathbf{X}\sim\mathbf{Y}$
if and only if there is a continuous, strictly increasing, onto function
$\sigma$ from $[0,1]$ onto $[0,1]$ such that 
\[
\mathbf{X}_{t}=\text{\ensuremath{\mathbf{Y}}}_{\sigma(t)}.
\]
 (i.e. $\mathbf{X}$ and $\mathbf{Y}$ are reparametrization of each
other. This relation was considered in depth in \cite{BG15}, Section
5.3).

Lemma 4.6 in \cite{BGLY16} states that two tree-reduced geometric
rough path have the same signature if and only if they are in the
same equivalence class. Corollary \ref{cor: uniqueness of signature for Brownian rough path}
therefore implies that for any two $\omega_{1}$ and $\omega_{2}$
outside some $\mathbb{P}$-null set $\mathcal{N}$, $\mathbf{B}(\omega_{1})$
and $\mathbf{B}(\omega_{2})$ have the same signature if and only
if they are in the same equivalence class.

Pick an arbitrary representative $(\mathbf{X}_{t})_{0\leqslant t\leqslant1}\in[\mathbf{B}(\omega)].$
Then 
\[
\mathbf{X}_{t}=\mathbf{B}_{\sigma(t)}(\omega),\ \ \ 0\leqslant t\leqslant1,
\]
for some unique reparametrization $\sigma$ that we want to figure
out. According to Proposition \ref{prop: universality of null set},
we have 
\[
\sigma(t)=\frac{1}{\kappa_{d}}\limsup_{n\rightarrow\infty}\left(\left(\frac{n}{2}\right)!\|\mathbb{X}_{0,t}^{n}\|_{\mathrm{proj}}\right)^{\frac{2}{n}},
\]
where we again choose the projective tensor norms on the tensor products
associated with the $l^{2}$-norm on $\mathbb{R}^{d}.$ (Recall that
by Theorem \ref{thm: lower estimate of constant}, $\kappa_{d}\neq0$)
The underlying path $\mathbf{B}(\omega)$ is then given by 
\[
\mathbf{B}_{t}(\omega)=\mathbf{X}_{\sigma^{-1}(t)},\ \ \ 0\leqslant t\leqslant1.
\]
\end{proof}
Another way of understanding the previous result is the following.
Since $[\mathbf{B}(\omega)]$ can be recovered from its signature,
we know that the image of the signature path $\mathbb{B}(\omega)$
can be recovered from its endpoint. For every tensor element $\xi=(1,\xi_{1},\xi_{2},\cdots)$
which can be realized as the signature of some Brownian sample path,
we then have 
\[
\mathbf{B}_{\|\xi\|/\kappa_{d}}(\omega)=\pi^{(2)}(\xi),
\]
where 
\[
\|\xi\|\triangleq\limsup_{n\rightarrow\infty}\left(\left(\frac{n}{2}\right)!\|\xi_{n}\|_{\mathrm{proj}}\right)^{\frac{2}{n}}
\]
and $\pi^{(2)}:\ T((\mathbb{R}^{d}))\rightarrow T^{(2)}((\mathbb{R}^{d}))$
is the canonical projection map.

Beyond the study of signature, Corollary \ref{cor: uniqueness of signature for Brownian rough path}
also gives the following property of Brownian rough path, which we
are unable to find in the literature.
\begin{cor}
\label{cor: no two BRPs can be equal up to reparametrization}There
exists a $\mathbb{P}$-null set $\mathcal{N},$ such that for any
two distinct $\omega_{1},\omega_{2}\notin\mathcal{N},$ $\mathbf{B}(\omega_{1})$
and $\mathbf{B}(\omega_{2})$ cannot be equal up to a reparametrization. 
\end{cor}
\begin{proof}
We follow the same notation as in the proof of Corollary \ref{cor: uniqueness of signature for Brownian rough path}.
Given two distinct $\omega_{1},\omega_{2}\notin\mathcal{N},$ suppose
that 
\[
\mathbf{B}_{t}(\omega_{2})=\mathbf{B}_{\sigma(t)}(\omega_{1}),\ \ \ 0\leqslant t\leqslant1,
\]
for some reparametrization $\sigma:\ [0,1]\rightarrow[0,1].$ Then
the signature of $t\rightarrow\mathbf{B}_{\sigma(t)}$ is equal to
$S(\mathbf{B})_{\sigma(t)}$ (see Lemma 1.4 in \cite{BGLYAnnex})
we have 
\[
\widetilde{L}_{0,\sigma(t)}(\omega_{1})=\kappa_{d}\sigma(t)
\]
and 
\[
\widetilde{L}_{0,t}(\omega_{2})=\kappa_{d}t.
\]
But from assumption we know that $\widetilde{L}_{0,\sigma(t)}(\omega_{1})=\widetilde{L}_{0,t}(\omega_{2}).$
Therefore, we must have $\sigma(t)=t$ and hence $\mathbf{B}(\omega_{1})=\mathbf{B}(\omega_{2})$. 
\end{proof}

\section{\label{sec:Further-remarks-and}Further remarks and related problems}

In Theorem \ref{thm: main result}, we considered the tail asymptotics
of the Brownian signature defined in terms of iterated Stratonovich's
integrals. Stratonovich's integrals arise naturally when we study
Brownian motion from the rough path point of view. On the other hand,
one could ask a similar question for Itô's iterated integrals. Indeed,
if we define 
\begin{equation}
\widehat{L}_{s,t}\triangleq\limsup_{n\rightarrow\infty}\left(\left(\frac{n}{2}\right)!\left\Vert \int_{s<u_{1}<\cdots<u_{n}<t}\mathrm{d}B_{u_{1}}\otimes\cdots\otimes\mathrm{d}B_{u_{n}}\right\Vert _{l^{1}}\right)^{\frac{2}{n}}\label{eq:lim sup in l1 norm}
\end{equation}
where the iterated integrals are defined in the sense of Itô and the
tensor products are equipped with the $l^{1}$-norm, then by a similar
type of arguments, one can show that for each $s<t$ 
\begin{equation}
\frac{d(t-s)}{2}\leqslant\widehat{L}_{s,t}\leqslant\frac{d^{2}(t-s)}{2}\label{eq: upper and lower estimates for Ito's signature}
\end{equation}
almost surely. Since the lifting of Brownian motion in \textit{Itô's}
sense is not a geometric rough path, uniqueness of signature result
does not apply and the intrinsic meaning of the quantity $\widehat{L}_{s,t}$
is unclear. The proof of (\ref{eq: upper and lower estimates for Ito's signature})
will not be included here since it is essentially parallel to the
Stratonovich case.

Our main result of Theorem \ref{thm: main result} gives rise to many
interesting and related problems in the probabilistic context.

(1) The first interesting and immediate question one could come up
with is the exact value of $\kappa_{d}$ and its probabilistic meaning.
In view of the length conjecture (\ref{eq: the length conjecture})
and Theorem \ref{thm: main result}, if we consider the projective
tensor norms on the tensor products induced by the Euclidean norm
on $\mathbb{R}^{d},$ it is quite natural to expect that, $\kappa_{d}$
would have a meaning related to certain kind of quadratic variation
for the Brownian rough path. It also seems that there are rooms for
improving the upper estimate for $\kappa_{d}.$ The point is that
in the proof of Lemma \ref{lem: estimating the sup-L^1 norm for the Brownian signature},
if we shuffle an arbitrary long word $\{i_{1},\cdots,i_{n}\}$ over
$\{1,\cdots,d\}$ with itself, the chance of hitting a nonzero coefficient
in the $2n$-degree component of the Brownian expected signature is
quite small. But to make the analysis precise, some hard combinatorics
argument for the shuffle product structure might be involved.

(2) If $\kappa_{d}$ is related to certain kind of quadratic variation
for the Brownian motion, it is reasonable to expect that our main
result and corollaries apply to diffusions or even general continuous
semimartingales, though there is no reason to believe that in this
case the corresponding $\widetilde{L}_{s,t}$ will still be deterministic.
For Gaussian processes, it is even not clear that any analogous version
of $\widetilde{L}_{s,t}$ would be meaningful since for instance we
know that 
\[
\lim_{n\rightarrow\infty}\sum_{i=1}^{n}|B_{\frac{i}{n}}-B_{\frac{i-1}{n}}|^{p}=0\ \ \mathrm{or}\ \ \infty
\]
in probability for a fractional Brownian motion with Hurst parameter
$H\in(0,1),$ according to whether $pH>1$ or $pH<1.$

(3) There is a quite subtle point in the discussion of Section 6.
With probability one, the lifting map $\omega\mapsto\mathbf{B}(\omega)$
is canonically well-defined. Therefore, although Corollary \ref{cor: strong uniqueness result for BRP}
(the uniqueness result) is stated at the level of the Brownian rough
path, by projection to degree one, it also holds at the level of sample
paths.

However, it is not at all clear if the first part of Corollary \ref{cor: no two BRPs can be equal up to reparametrization}
is true at the level of Brownian sample paths. More precisely, to
our best knowledge, a solution to the following classical question
for Brownian motion is not known (at least not to us yet): does there
exist a $\mathbb{P}$-null set $\mathcal{N},$ such that no two sample
paths outside $\mathcal{N}$ can be equal up to a non-trivial reparametrization?
This question is stated for Brownian sample paths and has nothing
to do with the lifting of Brownian motion to rough paths.

It is a subtle point that the result of Corollary \ref{cor: no two BRPs can be equal up to reparametrization}
does not yield an affirmative answer easily to the above question.
Indeed, if one wants to apply Corollary \ref{cor: no two BRPs can be equal up to reparametrization},
a missing point is whether the lifting operation and the reparametrization
operation are commutative outside some universal $\mathbb{P}$-null
set. In other words, it is not known if there exists a $\mathbb{P}$-null
set $\mathcal{N},$ such that one could define a lifting map $\omega\mapsto\mathbf{B}(\omega)$
for all $\omega\notin\mathcal{N},$ which satisfies 
\[
\mathbf{B}_{\cdot}(\omega_{\sigma})=\mathbf{B}_{\sigma(\cdot)}(\omega)
\]
for all reparametrizations $\sigma:\ [0,1]\rightarrow[0,1].$ When
defining the almost sure lifting of Brownian motion, the $\mathbb{P}$-null
set comes with the given choice of approximation. It is quite subtle
(and could be false) to see if the $\mathbb{P}$-null set can be chosen
in a universal way. 

(4) A final remark is about whether the limsup in (\ref{eq:lim sup in l1 norm})
can be replaced by sup. This is true for bounded variations, the proof
of which we now briefly explain. 

Let $g=(1,g_{1},g_{2},\cdots)$ be a group-like element. From the
shuffle product formula, 
\[
g_{k}^{\otimes n}=\sum_{\sigma\in\mathcal{S}(k,\cdots,k)}\mathcal{P}^{\sigma}(g_{nk}).
\]
Therefore, 
\[
\|g_{k}\|_{\mathrm{proj}}^{n}\leqslant\frac{(nk)!}{(k!)^{n}}\|g_{nk}\|_{\mathrm{proj}}.
\]
It follows that 
\[
\left(k!\|g_{k}\|_{\mathrm{proj}}\right)^{\frac{1}{k}}\leqslant\left((nk)!\|g_{nk}\|_{\mathrm{proj}}\right)^{\frac{1}{nk}},\ \ \ \forall n,k\geqslant1.
\]
Therefore, we conclude that 
\[
\sup_{n\geqslant1}\left(n!\|g_{n}\|_{\mathrm{proj}}\right)^{\frac{1}{n}}=\limsup_{n\rightarrow\infty}\left(n!\|g_{n}\|_{\mathrm{proj}}\right)^{\frac{1}{n}}.
\]
This is indeed true for any given admissible norms. A similar statement
with a fractional factorial normalization (which naturally corresponds
to the rough path situation) is \textit{not} true. Indeed, considering
the Brownian motion case, we have 
\[
\sup_{n\geqslant1}\left(\left(\frac{n}{2}\right)!\|\mathbb{B}_{0,1}^{n}\|_{\mathrm{proj}}\right)^{\frac{2}{n}}\geqslant\left(\left(\frac{1}{2}\right)!\|B_{1}-B_{0}\|_{\mathbb{R}^{d}}\right)^{2},
\]
while on the other hand, by Theorem \ref{thm: the limsup is a constant},
\[
\limsup_{n\rightarrow\infty}\left(\left(\frac{n}{2}\right)!\|\mathbb{B}_{0,1}^{n}\|_{\mathrm{proj}}\right)^{\frac{2}{n}}=\kappa_{d}
\]
for almost surely. Therefore, with positive probability the ``sup''
is not equal to the ``limsup'' for the Brownian signature.

\end{document}